\documentclass[leqno,12pt]{article} %leqno is the option to put formula numbers on the left side
\usepackage{amsmath, amssymb}
\usepackage{amsthm} %theorem environment option
\usepackage{pictexwd,dcpic}
\usepackage{xfrac}

\theoremstyle{plain} %text of this environment is typesetted in italics
\newtheorem{theorem}{\indent\sc Theorem}[section]
\newtheorem{lemma}[theorem]{\indent\sc Lemma}
\newtheorem{corollary}[theorem]{\indent\sc Corollary}
\newtheorem{proposition}[theorem]{\indent\sc Proposition}

\theoremstyle{definition} %text of this environment is typesetted in roman letters

\newtheorem{remark}[theorem]{\indent\sc Remark}

\newtheorem{notation}[theorem]{\indent\sc Notation}

\numberwithin{equation}{section}

\def\Z{\zeta}
\def\V{\varphi}
\def\O{\omega}
\def\E{\eta}
\def\G{\Gamma}
\def\Q{\mathbb{Q}}
\def\T{\Theta}

\def\X{\xi}

\makeatletter
\def\address#1#2{\begingroup
\noindent\parbox[t]{7.8cm}{%
\small{\scshape\ignorespaces#1}\par\vskip1ex
\noindent\small{\itshape E-mail address}%
\/: #2\par\vskip4ex}\hfill%
\endgroup}%

\def\Ddots{\mathinner{\mkern1mu\raise\p@
 \vbox{\kern7\p@\hbox{.}}\mkern2mu
 \raise4\p@\hbox{.}\mkern2mu\raise7\p@\hbox{.}\mkern1mu}}

\def\bbordermatrix#1{\begingroup \m@th
  \@tempdima 4.75\p@
  \setbox\z@\vbox{%
    \def\cr{\crcr\noalign{\kern2\p@\global\let\cr\endline}}%
    \ialign{$##$\hfil\kern2\p@\kern\@tempdima&\thinspace\hfil$##$\hfil
      &&\quad\hfil$##$\hfil\crcr
      \omit\strut\hfil\crcr\noalign{\kern-\baselineskip}%
      #1\crcr\omit\strut\cr}}%
  \setbox\tw@\vbox{\unvcopy\z@\global\setbox\@ne\lastbox}%
  \setbox\tw@\hbox{\unhbox\@ne\unskip\global\setbox\@ne\lastbox}%
  \setbox\tw@\hbox{$\kern\wd\@ne\kern-\@tempdima\left[\kern-\wd\@ne
    \global\setbox\@ne\vbox{\box\@ne\kern2\p@}%
    \vcenter{\kern-\ht\@ne\unvbox\z@\kern-\baselineskip}\,\right]$}%
  \null\;\vbox{\kern\ht\@ne\box\tw@}\endgroup}

\makeatother
\title{{Construction of class fields over cyclotomic fields}} %title of the paper
\author{
\textsc{Ja Kyung Koo and Dong Sung Yoon} %names of authors
}

\date{} %leave empty
\begin{document}

\maketitle

%%%%%%%%%%%%%%% footnote %%%%%%%%%%%%%%%%
\footnote{ %2010 MSC numbers
 }
\footnote{ %key words and phrases
    } 
\footnote{\thanks{
The second named author was supported by the National Institute for Mathematical Sciences, Republic of Korea.} }
%%%%%%%%%%%%%%%%%%%%%%%%%%%%%%%%%%%%%%%%%$$$$$$$$$$$$$$$$$$$$$$$$$$$$$$$$$$$$$$$$$$$$$$$$$$$$$$$$$$$$$$$$$$$$$$$$$$$$$$$$$$$$

\begin{abstract}
Let $\ell$ and $p$ be odd primes. For a positive integer $\mu$ let $k_\mu$ be the ray class field of $k=\mathbb{Q}(e^{2\pi i/\ell})$ modulo $2p^\mu$.
We present certain class fields $K_\mu$ of $k$ such that $k_\mu\subset K_\mu\subset k_{\mu+1}$, and provide a necessary and sufficient condition for $K_\mu=k_{\mu+1}$.
And we also construct, in the sense of Hilbert, primitive generators of the field $K_\mu$ over $k_\mu$ by using Shimura's reciprocity law and special values of theta constants. 
\end{abstract}

\maketitle

\section{Introduction}

In his 12th problem (1900 Paris ICM) Hilbert asked that what kind of analytic functions and algebraic numbers are necessary to construct all abelian extensions of given number fields.
For any number field $K$ and a modulus $\mathfrak{m}$ of $K$, it is well known (\cite{Takagi} or \cite[Theorem 8.6]{Cox}) that there is a unique maximal abelian extension of $K$ unramified outside $\mathfrak{m}$, which is called the \textit{ray class field} of $K$ modulo $\mathfrak{m}$.
Hence, as a first step toward the problem we need to construct ray class fields for given number fields.
Historically, over imaginary quadratic number fields $K$, Hasse(\cite{Hasse}) constructed the ray class field of $K$ by making use of the Weber function and the elliptic modular function.
And, after Hasse many people investigated in this theme, for example, see \cite{Cais}, \cite{Chen}, \cite{Cho}, \cite{Cougnard}, \cite{Cougnard2}, \cite{Cox}, \cite{Cox2}, \cite{Eichler}, \cite{Jung}, \cite{Ramachandra}, \cite{Schertz}, \cite{Schertz2}, \cite{Schertz3} and \cite{Stevenhagen}.
On the other hand, over any other CM-fields $K$ with $[K:\mathbb{Q}]>2$, not much seems to be known so far.
For instance, over a cyclotomic number field $K$ with odd relative class number, Shimura(\cite{Shimura5}) showed that the Hilbert class field of $K$ is generated by that of the maximal real subfield of $K$ and the unramified abelian extensions of $K$ obtained by the fields of moduli of certain two polarized abelian varieties having subfields of $K$ as endomorphism algebras.
And, by making use of Galois representation Ribet(\cite{Ribet}) constructed unramified abelian, degree $p$-extensions of $K=\mathbb{Q}(e^{2\pi i/p})$ for all irregular primes $p$. 
See, also \cite{Mazur}.
Furthermore, Komatsu(\cite{Komatsu}) investigated a certain class field of $K=\mathbb{Q}(e^{2\pi i/5})$  and constructed its normal basis by means of Siegel modular functions.
\par
Now, let $n$ be a positive integer, $k$ be a CM-field with $[k:\mathbb{Q}]=2n$, $k^*$ be its reflex field and $z_0$ be the associated CM-point (\S \ref{Shimura's reciprocity law}).
Shimura showed in \cite{Shimura2} that if $f$ is a Siegel modular function which is finite at $z_0$, then the special value $f(z_0)$ belongs to some abelian extension (= class field) of $k^*$. 
And, his reciprocity law explains Galois actions on $f(z_0)$ in terms of action of the group $G_{\mathbb{A}+}$ on $f$ (Proposition \ref{reciprocity}). 
Here $G_{\mathbb{A}+}=\prod_p'\textrm{GSp}_{2n}(\mathbb{Q}_p)\times \textrm{GSp}_{2n}^+(\mathbb{R})$ is the restricted product with respect to the subgroups $\textrm{GSp}_{2n}(\mathbb{Z}_p)$ of $\textrm{GSp}_{2n}(\mathbb{Q}_p)$.
He also constructed in \cite{Shimura1} Siegel modular functions by the quotient of two theta constants 
\begin{equation*}
\Phi_{(r,s)}(z)=\frac{\sum_{x\in \mathbb{Z}^n}e\Big(\frac{1}{2}~{^t}(x+r)z(x+r)+{^t}(x+r)s\Big)}{\sum_{x\in \mathbb{Z}^n}e\Big(\frac{1}{2}~{^t}xzx\Big)}
\end{equation*}
for $r,s\in\mathbb{Q}^n$, and explicitly describe the Galois actions on the special values of theta functions 
(\S \ref{Theta functions}). 
\par
In this paper, we mainly consider the case of cyclotomic number field $k=\mathbb{Q}(e^{2\pi i/\ell})$ for any odd prime $\ell$.
Let $p$ be an odd prime and $\mu$ be a positive integer. 
We denote by $k_\mu$ the ray class field of $k$ modulo $2p^\mu$.
In Section \ref{Class fields over cyclotomic fields}, we define the class field $K_\mu$ of $k$ such that $k_\mu\subset K_\mu\subset k_{\mu+1}$, which would be an extension of Komatsu's result (\cite{Komatsu}). 
We shall first find the exact degree of $K_\mu$ over $k_\mu$ for any odd prime $\ell$ (Theorem \ref{dimension}).
And, we shall further provide a necessary and sufficient condition for $K_\mu$ to be the ray class field $k_{\mu+1}$ (Corollary \ref{unit2}).
In Section \ref{Construction of class fields}, as Hilbert proposed, by using Shimura's reciprocity law we shall construct a primitive generator of $K_\mu/k_\mu$ in terms of special value of $\Phi_{(r,s)}(z)$ for some $r,s\in \mathbb{Q}^n$ at the CM-point corresponding to the polarized abelian variety of genus $n=(\ell-1)/{2}$ (Theorem \ref{generator}).

\begin{notation}
For $z\in\mathbb{C}$, we denote by $\overline{z}$ the complex conjugate of $z$ and by Im$(z)$ the imaginary part of $z$, and put $e(z)=e^{2\pi i z}$.
If $R$ is a ring with identity and $r,s\in\mathbb{Z}_{>0}$, $M_{r\times s}(R)$ indicates the ring of all $r\times s$ matrices with entries in $R$. 
In particular, we set $M_{r}(R)=M_{r\times r}(R)$.
The identity matrix of $M_{r}(R)$ is written by $1_r$ and the transpose of a matrix $\alpha$ is denoted by ${^t}\alpha$. And, $R^\times$ stands for the group of all invertible elements of $R$.
If $G$ is a group and $g_1,g_2,\ldots,g_r$ are elements of $G$, let $\langle g_1,g_2,\ldots,g_r \rangle$ be the subgroup of $G$ generated by $g_1,g_2,\ldots,g_r$, and $G^n$ be the subgroup $\{g^n~|~g\in G \}$ of G for $n\in\mathbb{Z}_{>0}$.
Moreover, if $H$ is a subgroup of $G$, let $|G:H|$ be the index of $H$ in $G$.
For a finite algebraic extension $K$ over $F$, $[K:F]$ denotes the degree of $K$ over $F$.
We let $\Z_N=e^{2\pi i/N}$ be a primitive $N$th root of unity for a positive integer $N$.
\end{notation}
\par

\section{Siegel modular forms}
We shall briefly present necessary facts about Siegel modular forms and explain the action of $G_{\mathbb{A}+}$ on the Siegel modular functions whose Fourier coefficients lie in some cyclotomic fields.
\par

Let $n$ be a positive integer and $G$ be the algebraic subgroup of $\mathrm{GL}_{2n}$ defined over $\mathbb{Q}$ such that
\begin{equation*}
G_\mathbb{Q}=\big\{\alpha\in \mathrm{GL}_{2n}(\mathbb{Q})~|~{^t}\alpha J\alpha=\nu(\alpha)J ~\textrm{ with $\nu(\alpha)\in\mathbb{Q}^\times$}\big\},
\end{equation*}
where
\begin{equation*}
J=J_n=\left[
\begin{matrix}
0&-1_n\\
1_n&0
\end{matrix}\right].
\end{equation*}
Let $G_{\mathbb{A}}$ be the adelization of $G$, $G_0$ the non-archimedean part of $G_\mathbb{A}$, and $G_\infty$ the archimedean part of $G_\mathbb{A}$. 
We extend the multiplier map $\nu:G_\mathbb{Q}\rightarrow \mathbb{Q}^\times$ to a continuous map of $G_\mathbb{A}$ into $\mathbb{Q}_\mathbb{A}^\times$, which we denote again by $\nu$. Then we put $G_{\infty+}=\{x\in G_\infty~|~\nu(x)\gg 0\}$ and $G_{\mathbb{A}+}=G_0G_{\infty+}$. 
Here $t\gg 0$ means $t_v>0$ for all archimedean primes $v$ of $\mathbb{Q}$.
For a positive integer $N$, let 
\begin{equation*}
\begin{array}{ccl}
R_N&=&\mathbb{Q}^\times\cdot\{a\in G_{\mathbb{A}+}~|~ a_q\in \mathrm{GL}_{2n}(\mathbb{Z}_q), a_q\equiv 1_{2n}~ (\bmod~ N\cdot M_{2n}(\mathbb{Z}_q))\textrm{ for all primes $q$}\},\vspace{0.1cm}\\
\Delta&=&\displaystyle\bigg\{
\left[\begin{matrix}
1_n&0\\
0&x\cdot 1_n
\end{matrix}\right]
~|~x\in\prod_q \mathbb{Z}_q^\times\bigg\},\vspace{0.1cm}\\
G_{\Q+}&=&\{\alpha\in G_\Q~|~\nu(\alpha)>0 \}.
\end{array}
\end{equation*}
\begin{proposition}
For every positive integer $N$, we have
\begin{equation*}
G_{\mathbb{A}+}=R_N\Delta G_{\mathbb{Q}+}.
\end{equation*}
\end{proposition}
\begin{proof}
\cite[Proposition 3.4]{Shimura3} and \cite[p.535 (3.10.3)]{Shimura4}.
\end{proof}

Let $\mathbb{H}_n=\{z\in M_n(\mathbb{C})~|~{^t}z=z,~\mathrm{Im}(z)>0\}$ be the Siegel upper half-space of degree $n$.
Here, for a hermitian matrix $\xi$ we write $\xi>0$ to mean that $\xi$ is positive definite.
We define the action of an element $\alpha=
\left[\begin{matrix}
A&B\\
C&D
\end{matrix}\right]$
of $G_{\Q+}$ on $\mathbb{H}_n$ by
\begin{equation*}
\alpha(z)=(Az+B)(Cz+D)^{-1},
\end{equation*}
where $A,B,C,D\in M_n(\mathbb{Q})$.
For every positive integer $N$, let 
\begin{equation*}
\G(N)=\big\{\gamma \in \mathrm{Sp}_{2n}(\mathbb{Z})~|~\gamma\equiv 1_{2n}\pmod{N\cdot M_{2n}(\mathbb{Z})} \big\}.
\end{equation*} 
For an integer $m$, a holomorphic function 
$f:\mathbb{H}_n\rightarrow\mathbb{C}$ is called a (classical) \textit{Siegel modular form of weight $m$ and level $N$} if
\begin{eqnarray}\label{modularity}
f(\gamma (z))=\det(Cz+D)^m f(z) 
\end{eqnarray}
for all $\gamma=\left[\begin{matrix}A&B\\C&D\end{matrix}\right] \in \G (N)$ and $z\in \mathbb{H}_n$, plus  the requirement when $n=1$ that $f$ is holomorphic at every cusp. 
In particular, $f(z)$ has a Fourier expansion of the form
\begin{equation*}
f(z)=\sum_{\xi}A(\xi)e(tr(\xi z)/N)
\end{equation*}
with $A(\xi)\in \mathbb{C}$, where $\xi$ runs over all positive semi-definite half-integral symmetric matrices of degree $n$ \cite[\S4 Theorem 1]{Klingen}. 
Here, a symmetric matrix $\xi \in M_n(\Q)$ is called half-integral if $2\xi$ is an integral matrix whose diagonal entries are even.
\par
For a subring $R$ of $\mathbb{C}$, let $\mathcal{M}_m\big(\G(N),R\big)$ be the vector space of all Siegel modular forms $f$ of weight $m$ and level $N$ whose Fourier coefficients $A(\xi)$ belong to $R$ and let $\mathcal{M}_m(R)=\displaystyle\bigcup_{N=1}^\infty \mathcal{M}_m\big(\G(N),R\big)$. 
We denote by $\mathcal{A}_m(R)$ the set of all meromorphic functions of the form $g/h$ with $g\in\mathcal{M}_{r+m}(R),~0\neq h\in\mathcal{M}_{r}(R)$ (with any $r\in\mathbb{Z}$), and by $\mathcal{A}_{m}\big(\G(N),R\big)$ the set of all $f\in \mathcal{A}_{m}(R)$ satisfying (\ref{modularity}).
In particular, we set 
\begin{eqnarray*}
\mathcal{F}_N&=&\mathcal{A}_{0}\big(\G(N),\mathbb{Q}(\Z_N)\big),\\
\mathcal{F}&=&\bigcup_{N=1}^\infty \mathcal{F}_N.
\end{eqnarray*}
\par
For every algebraic number field $K$, let $K_{ab}$ be the maximal abelian extension of $K$, and $K_\mathbb{A}^\times$ the idele group of $K$. 
By class field theory, every element $x$ of $K_\mathbb{A}^\times$ acts on $K_{ab}$ as an automorphism. 
We then denote this automorphism by $[x,K]$. 
On the other hand, every element of $G_{\mathbb{A}+}$ acts on $\mathcal{F}$ as an automorphism (\cite[p.680]{Shimura1}). 
If $x\in G_{\mathbb{A}+}$ and $f\in \mathcal{F}$, we denote by $f^x$ the image of $f$ under $x$.

\begin{proposition}\label{Siegel-action}
Let $f(z)=\sum_{\xi}A(\xi)e(tr(\xi z)/N)\in\mathcal{F}_N$.
Then we get the followings:
\begin{itemize}
\item[\textup{(i)}] $f^\beta=f$ for $\beta\in R_N$. 
Moreover, $\mathcal{F}_N$ is the subfield of $\mathcal{F}$ consisting of all the $R_N$-invariant elements.
\item[\textup{(ii)}] 
Let $y=
\left[\begin{matrix}
1_n&0\\
0&x\cdot 1_n
\end{matrix}\right]
\in\Delta$ and $t$ be a positive integer such that $t\equiv x_q \pmod{N\mathbb{Z}_q}$
for all primes $q$.
Then we derive
\begin{equation*}
f^y=\sum_{\xi}A(\xi)^\sigma e(tr(\xi z)/N),
\end{equation*} 
where $\sigma$ is the automorphism of $\mathbb{Q}(\Z_N)$ such that $\Z_N^\sigma=\Z_N^t$.
\item[\textup{(iii)}] $f^\alpha=f\circ\alpha$ for $\alpha\in G_{\mathbb{Q}+}$.
\end{itemize}
\end{proposition}
\begin{proof}
\cite[p.681]{Shimura1} and \cite[Theorem 26.8]{Shimura2}.
\end{proof}

\section{Shimura's reciprocity law}\label{Shimura's reciprocity law}
We begin with fundamental but necessary facts about Shimura's reciprocity law \cite[\S 26]{Shimura2}.
\par
Let $n$ be a positive integer, K be a CM-field with $[K:\mathbb{Q}]=2n$ and $\mathcal{O}_K$ be a ring of integers of $K$.
And, let $\V_1, \V_2, \ldots, \V_n$ be $n$ distinct embeddings of $K$ into $\mathbb{C}$ such that there are no two embeddings among them which are complex conjugate of each other on $K$.
Then $(K;\{\V_1, \V_2, \ldots, \V_n\})$ is a CM-type and
we can take an element $\rho$ in $K$ such that
\begin{itemize}
\item[\textup{(i)}] $\rho$ is purely imaginary, 
\item[\textup{(ii)}] $-\rho^2$ is totally positive,
\item[\textup{(iii)}] Im$(\rho^{\V_i})>0$ for all $i=1,\ldots,n$,
\item[\textup{(iv)}] Tr$_{K/\mathbb{Q}}(\rho \xi)\in\mathbb{Z}$ for all $\xi\in\mathcal{O}_K$.
\end{itemize}
We denote by $v(x)$, for $x\in K$, the vector of $\mathbb{C}^n$ whose components are $x^{\V_1},\ldots,x^{\V_n}$.
The set $L=\big\{v(x)~\big|~x\in\mathcal{O}_K\big\}$ is a lattice in $\mathbb{C}^n$.
We define an $\mathbb{R}$-bilinear form $E(z,w)$ on $\mathbb{C}^n$ by
\begin{equation*}
E(z,w)=\sum_{i=1}^{n} \rho^{\V_i}(z_i\overline{w_i}-\overline{z_i}w_i)\quad\textrm{for $z=\left[\begin{matrix}
z_1\\
\vdots\\
z_n
\end{matrix}\right]$ and $w=\left[\begin{matrix}
w_1\\
\vdots\\
w_n
\end{matrix}\right]$}.
\end{equation*}
Then $E$ becomes a non-degenerate Riemann form on the complex torus $\mathbb{C}^n/L$ satisfying
\begin{equation*}
E\big(v(x),v(y)\big)=\mathrm{Tr}_{k/\mathbb{Q}}(\rho x\overline{y}) \textrm{ ~~for $x,y \in K$},
\end{equation*}
which makes it a polarized abelian variety (\cite[p.43--44]{Shimura2}).
Hence we can find a positive integer $\delta$, a diagonal matrix $\epsilon$ with integral elements, and a complex $(n\times 2n)$-matrix $\Omega$ such that
\begin{itemize}
\item[\textup{(i)}] $E(\Omega x, \Omega y)= \delta\cdot{^t}xJy ~~~\textrm{for $x,y\in\mathbb{R}^{2n}$}$,
\item[\textup{(ii)}] $L=\Big\{\Omega\left[\begin{matrix} a\\b \end{matrix}\right]~\big|~a\in\mathbb{Z}^n, b\in\epsilon\mathbb{Z}^n \Big\}$,
\item[\textup{(iii)}] $\epsilon=\left[\begin{matrix} \epsilon_1&&&\\ &\epsilon_2&& \\ &&\ddots& \\ &&& \epsilon_n\end{matrix}\right],
~\epsilon_1=1,~\epsilon_i~|~\epsilon_{i+1}~~~\textrm{for $i=1,\ldots,n-1$}$.
\end{itemize}
(\cite[Lemma 27.2]{Shimura2} or \cite[p.675]{Shimura1}).
Now, we write $\Omega=\big[\Omega_1~~\Omega_2\big]=\big[v(e_1)~~v(e_2)~~\cdots~~v(e_{2n})\big]$ with $\Omega_1,\Omega_2\in M_n(\mathbb{C})$ and $e_1, e_2, \ldots, e_{2n}\in K$, and put $z_0=\Omega_2^{-1}\Omega_1$. 
It is well-known that $z_0\in\mathbb{H}_n$. 
Let $\Phi:K\rightarrow M_n(\mathbb{C})$ be a ring monomorphism such that
\begin{equation*}
\Phi(x)=\left[\begin{matrix}
x^{\V_1}&&& \\
&x^{\V_2}&& \\
&&\ddots& \\
&&&x^{\V_n}
\end{matrix}\right]
~~\textrm{for $x\in K$}.
\end{equation*}
Then we can define a ring monomorphism $h:K\rightarrow M_{2n}(\mathbb{Q})$ by
\begin{equation*}
\Phi(x)\Omega=\Omega\cdot {^t}h(x)~~~ \textrm{for $x\in K$}.
\end{equation*}
Here, $h(x)=\big[a_{ij}\big]_{1\leq i,j \leq 2n}$ is in fact the regular representation of $x$ with respect to $\{e_1, e_2,\ldots,e_{2n}\}$, namely $x e_i=\displaystyle\sum_{j=1}^{2n}a_{ij}e_j$.
If $\epsilon=1_n$, then $L=v(\mathcal{O}_K)=\Omega\cdot\mathbb{Z}^{2n}=\mathbb{Z}v(e_1)+\cdots+\mathbb{Z}v(e_{2n})$ so that $h(x)\in M_{2n}(\mathbb{Z})$ for $x\in\mathcal{O}_K$.
One can then readily show that 
\begin{equation*}
h(\overline{x})=J~{^t} h(x)J^{-1}~~\textrm{ for $x\in K$},
\end{equation*}
and $z_0$ is the CM-point of $\mathbb{H}_n$ induced from $h$ which corresponds to the principally polarized abelian variety $(\mathbb{C}^n/L, E)$
(\cite[p.684-685]{Shimura1} or \cite[\S 24.10]{Shimura2}).
In particular, if we set $S=\{x\in K^\times ~|~x \overline{x}\in\mathbb{Q}^\times \}$ then
$h(S)=\{h(s)~|~s\in S\}=\{\alpha\in G_{\mathbb{Q}+}~|~\alpha(z_0)=z_0\}$.

\par
Let $K^*$ be the reflex field of $K$ and $K'$ be a Galois extension of $K$ over $\mathbb{Q}$, and extend $\V_i~(i=1,\ldots,n)$ to an element of Gal$(K'/\mathbb{Q})$, which we denote again by $\V_i$.
Let $\{\psi_j\}_{j=1}^m$ be the set of all embeddings of $K^*$ into $\mathbb{C}$ obtained from $\{\V_i^{-1}\}_{i=1}^n$. 

\begin{proposition}\label{reflex}
Let $K$, $K^*$ and $\{\psi_j\}$ be as above.
\begin{itemize}
\item[\textup{(i)}] $(K^*;\{\psi_1,\ldots,\psi_m\})$ is a primitive CM-type and we have
\begin{equation*}
K^*=\mathbb{Q}\Big(\sum_{i=1}^n x^{\V_i}~|~x\in K\Big).
\end{equation*}
\item[\textup{(ii)}] If $b=\prod_j a^{\psi_j}$ with $a\in K^*$, then $b\in K$ and $b\overline{b}=N_{K^*/\mathbb{Q}}(a)$.
\end{itemize}
\end{proposition}
\begin{proof}
\cite[p.62--63]{Shimura2}.
\end{proof}
We call the CM-type $(K^*;\{\psi_j\})$ the \textit{reflex} of $( K;\{ \V_i\} )$.
By Proposition \ref{reflex}, we can define a homomorphism $\V^*:(K^*)^\times\rightarrow K^\times$ by
\begin{displaymath}
\V^* (a)=\prod_{j=1}^{m}a^{\psi_j}~~ \mathrm{for}~ a\in (K^*)^\times,
\end{displaymath}
and we have $\V^*(a)\cdot\overline{\V^*(a)}=N_{K^*/\mathbb{Q}}(a)$ for $a\in (K^*)^\times$.  
The map $h$ can be extended naturally to a homomorphism $K_{\mathbb{A}}\rightarrow M_{2n}(\mathbb{Q}_\mathbb{A})$, which we also denote by $h$.
Then for every $b\in (K^*)_{\mathbb{A}}^\times$ we get $\nu\big(h(\V^*(b))\big)=N_{K^*/\mathbb{Q}}(b)$ and $h\big(\V^*(b)^{-1}\big)\in G_{\mathbb{A}+}$ (\cite[p.172]{Shimura2}).

\begin{proposition}[Shimura's reciprocity law]\label{reciprocity}
Let $K$, $h$, $z_0$ and $K^*$ be as above. Then for every $f\in\mathcal{F}$ which is finite at $z_0$, the value $f(z_0)$ belongs to $K^*_{ab}$. Moreover, if $b\in (K^*)_\mathbb{A}^\times$, then $f^{h(\V^*(b)^{-1})}$ is finite at $z_0$ and
\begin{equation*}
f(z_0)^{[b,K^*]}=f^{h(\V^*(b)^{-1})}(z_0).
\end{equation*}
\end{proposition}
\begin{proof}
\cite[Theorem 26.8]{Shimura2}.
\end{proof}

\begin{remark}\label{class field}
For any $f\in\mathcal{F}$ which is finite at $z_0$, the value $f(z_0)$ in fact belongs to the class field $\widetilde{K^*_{ab}}$ of $K^*$ corresponding to the kernel of $\varphi^*$.

\end{remark}

\section{Class fields over cyclotomic fields}\label{Class fields over cyclotomic fields}

Let $\ell$ and $p$ be odd prime numbers. 
We also write for simplicity $\Z=\Z_\ell$. 
Set $k=\mathbb{Q}(\Z)$ and $n=({\ell-1})/{2}$ so that $2n=[k:\mathbb{Q}]$. 
For $1\leq i \leq 2n$ we denote by $\varphi_i$ the element of $\mathrm{Gal}(k/\mathbb{Q})$ defined by $\Z^{\V_i}=\Z^i$. 
Then $( k;\{ \V_1,\V_2, \ldots, \V_n\} )$ is a primitive CM-type and $( k;\{ \V_1^{-1},\V_2^{-1}, \ldots, \V_n^{-1}\} )$ is its reflex (\cite[p.64]{Shimura2}).
For a positive integer $\mu$, put 
\begin{eqnarray*}
S_\mu&=&\{a\in k^\times ~|~ a\equiv 1\pmod{ 2p^\mu}\}\\
\widetilde{S_\mu}&=&\{(a) ~|~ a\in S_\mu \},
\end{eqnarray*}
where $(a)$ is the principal ideal of $k$ generated by $a$. 
Let $E$ be the unit group of $k$ and let $k_\mu$ be the ray class field of $k$ modulo $2p^\mu$. Then we have
\begin{equation*}
\mathrm{Gal}(k_{\mu+1}/k_{\mu})\cong \widetilde{S_\mu}/\widetilde{S_{\mu+1}}\cong S_\mu E/S_{\mu+1}E \cong S_\mu/S_{\mu+1}(S_\mu \cap E)
\end{equation*}
by class field theory. 
Further, we set $H_\mu=S_{\mu+1}(S_\mu \cap E)$ and
\begin{displaymath}
\O_{\mu,i} = \left\{ \begin{array}{ll}
1+2p^\mu\Z^i~ & \textrm{for $1\leq i \leq n+1$}\vspace{0.1cm}\\
1+2p^\mu(\Z^n+\Z^{n+1}-\Z^i-\Z^{-i})~ & \textrm{for $n+2\leq i\leq 2n$}.\\
\end{array} \right.
\end{displaymath}
Since the ring of integers $\mathcal{O}_k$ of $k$ is equal to $\mathbb{Z}[\Z]$ and $S_\mu/S_{\mu+1}$ is isomorphic to $\mathcal{O}_k/p\mathcal{O}_k$ by a mapping 
\begin{eqnarray*}
S_\mu/S_{\mu+1} \longrightarrow \mathcal{O}_k/p\mathcal{O}_k \textrm{\qquad\qquad\quad ~~~~~}\\
(1+2p^\mu\O)S_{\mu+1} \longmapsto \O+p\mathcal{O}_k \textrm{~~~ for $\O\in\mathcal{O}_k$},
\end{eqnarray*}
we obtain $S_\mu/S_{\mu+1}\cong (\mathbb{Z}/p\mathbb{Z})^{2n}$ and
\begin{equation*}
S_\mu/S_{\mu+1}=\langle(1+2p^\mu\Z)S_\mu, (1+2p^\mu\Z^2)S_\mu, \ldots, (1+2p^\mu\Z^{2n})S_\mu\rangle .
\end{equation*}
Let $B=[b_{ij}]\in M_{2n}(\mathbb{Z})$ where $b_{ij}$ is an integer such that $\O_{\mu,i}=1+2p^\mu\big(\sum_{j=1}^{2n}b_{ij}\Z^j\big)$.
Then we get
\begin{displaymath}
B=
\left[ \begin{array}{cccc|cccc}
1& & & &0 &0 &\cdots &0 \\
 &1& & &\vdots &\vdots & &\vdots  \\
& &\ddots&  &\vdots &\vdots & &\vdots\\
& &  &1 &0 &0 &\cdots &0\\
\hline
 & & & 0 &1&  & &\\
 &  &-1&1&1&-1& & \\
 &\Ddots& &\vdots&\vdots& &\ddots& \\
 -1& & &1&1& & & -1
\end{array} \right].
\end{displaymath}
Hence, $S_\mu/S_{\mu+1}=\langle\O_{\mu,1} S_{\mu+1}, \O_{\mu,2} S_{\mu+1}, \ldots, \O_{\mu,2n} S_{\mu+1}\rangle$ because $\det(B)=(-1)^{n-1}$ is prime to $p$.
This shows that $S_\mu/H_\mu=\langle\O_{\mu,1} H_\mu, \O_{\mu,2} H_\mu, \ldots, \O_{\mu,2n} H_\mu\rangle$ due to the fact $H_\mu\supset S_{\mu+1}$.
\par
Now, we define an endomorphism $\V^*$ of $k^\times$ by
\begin{displaymath}
\V^* (a)=\prod_{i=1}^{n}a^{\V_i^{-1}}~~ \mathrm{for}~ a\in k^\times
\end{displaymath}
and an endomorphism $\V^+$ of $k$ by 
\begin{displaymath}
\V^+(a)=\sum_{i=1}^n a^{\V_i^{-1}}  ~~\mathrm{for}~ a\in k.
\end{displaymath}
And, we let
\begin{displaymath}
\E_{\mu,i} = \left\{ \begin{array}{ll}
1+2p^\mu\V^+(\Z^i)~ & \textrm{for $1\leq i \leq n+1$}\\
1~ & \textrm{for $n+2\leq i\leq 2n$}\\
\end{array} \right.
\end{displaymath}
so that $\V^*(\O_{\mu,i})H_\mu=\E_{\mu,i}H_\mu$ for all $1\leq i\leq 2n$.
Since $\V^*(H_\mu)\subset H_\mu$, we can define an endomorphism $\widetilde{\V^*_\mu}$ of $S_\mu/H_\mu$ by $\widetilde{\V^*_\mu}(aH_\mu)=\V^*(a)H_\mu$.
Let $K_\mu$ be the class field of $k$ corresponding to the kernel of $\widetilde{\V^*_\mu}$.
Note that $K_\mu=k_\mu(k_{\mu+1}\cap \widetilde{k_{ab}})$, where $\widetilde{k_{ab}}$ is the class field of $k$ in Remark \ref{class field}.
Then we achieve
\begin{equation}\label{Galois}
\textrm{Gal}(K_\mu/k_\mu)\cong \sfrac{S_\mu/H_\mu}{ker(\widetilde{\V^*_\mu})}\cong \widetilde{\V^*_\mu}(S_\mu/H_\mu)=\langle\E_{\mu,1}H_\mu, \E_{\mu,2}H_\mu,\ldots,\E_{\mu,n+1}H_\mu\rangle.
\end{equation}
Observe that $K_\mu$ is the fixed field of $\Big\{\big(\frac{k_{\mu+1}/k}{(\O)}\big)~|~\O H_\mu\in ker(\widetilde{\V^*_\mu}) \Big\}$ and 
\begin{equation*}
\textrm{Gal}(K_\mu/k_\mu)=\bigg\langle\Big(\frac{k_{\mu+1}/k}{(\O_{\mu,1})}\Big), \Big(\frac{k_{\mu+1}/k}{(\O_{\mu,2})}\Big), \ldots, \Big(\frac{k_{\mu+1}/k}{(\O_{\mu,n+1})}\Big)\bigg\rangle. 
\end{equation*}
Here, $\big(\frac{k_{\mu+1}/k}{\cdot}\big)$ is the Artin map of $k_{\mu+1}/k$.

\begin{proposition}\label{cyclotomic1}
Let $N$ be a positive integer, $K=\mathbb{Q}(\Z_N)$ and $K^+$ be its maximal real subfield.
Let $E$ (resp. $E^+$) be the unit group of $K$ (resp. $K^+$) and $W$ be the group of roots of unity in $K$.
Then we have
\begin{displaymath}
|E:WE^+| = \left\{ \begin{array}{ll}
1~ & \textrm{if $N$ is a prime power}\\
2~ & \textrm{if $N$ is not a prime power.}\\
\end{array} \right.
\end{displaymath}
\end{proposition}
\begin{proof}
\cite[Corollary 4.13]{Washington}.
\end{proof}

\begin{proposition}\label{cyclotomic2}
Let $\ell$ be any prime and $m\in \mathbb{Z}_{>0}$. 
Let $\mathbb{Q}(\Z_{\ell^m})^+$ be the maximal real subfield of $\mathbb{Q}(\Z_{\ell^m})$ and $E_{\ell^m}^+$ be its unit group.
Further, we let $C_{\ell^m}^+$ be the subgroup of $E_{\ell^m}^+$ generated by $-1$ and the real units
\begin{equation*}
\xi_a=\Z_{2\ell^m}^{{1-a}}\cdot\frac{1-\Z_{\ell^m}^a}{1-\Z_{\ell^m}}\in\mathbb{R},~~1<a<\frac{\ell^m}{2},~~\gcd(a,\ell)=1.
\end{equation*}
Then
\begin{equation*}
h_{\ell^m}^+=|E_{\ell^m}^+:C_{\ell^m}^+|,
\end{equation*}
where $h_{\ell^m}^+$ is the class number of $\mathbb{Q}(\Z_{\ell^m})^+$.
\end{proposition}
\begin{proof}
\cite[Lemma 8.1 and Theorem 8.2]{Washington}.
\end{proof}

\begin{lemma}\label{unit}
Let $\ell$ be an odd prime and $p$ be an odd prime such that $p\nmid \ell h_{\ell}^+$, where $h_{\ell}^+$ is the class number of the maximal real subfield of $k$.
Then $H_\mu/S_{\mu+1}$ is generated by real units of $k$ for all $\mu\in\mathbb{Z}_{>0}$.
\end{lemma}
\begin{proof}
The $(2\ell h_\ell^+)$-th power mapping of $S_\mu/S_{\mu+1}$ induces an automorphism of itself because $\gcd(p, 2\ell h_{\ell}^+)=1$. Thus the image of $E\cap S_\mu$ in $S_\mu/S_{\mu+1}$ is the same as that of $E^{2\ell h_\ell^+}\cap S_\mu$.
By Proposition \ref{cyclotomic1} and \ref{cyclotomic2}, $E^{2\ell h_\ell^+}\subset\langle \xi_a^{2\ell}~ |~ 1 < a < {\ell}/{2} \rangle$ where $\xi_a=\Z_{2\ell}^{{1-a}}({1-\Z^a})/({1-\Z})\in\mathbb{R}$.  
Therefore $H_\mu/S_{\mu+1}=S_{\mu+1}(E^{2\ell h_\ell^+}\cap S_\mu)/S_{\mu+1}$ is generated by real units of $k$.
\end{proof}
Let $M_\ell(p)=[m_{ij}]\in M_{(n+1)\times 2n}(\mathbb{Z}/p\mathbb{Z})$ where $m_{ij}$ is the coefficient of $\Z^j$ in $\V^+(\Z^i)$ in $\mathbb{Z}/p\mathbb{Z}$. Then we get
\begin{displaymath}
m_{ij} = \left\{ \begin{array}{ll}
1~ & \textrm{if~ $\overline{i}\cdot\overline{j}^{~-1}\in \{\overline{1},\overline{2},\ldots,\overline{n}\}$ in $\mathbb{Z}/\ell\mathbb{Z}$}\\
0~ & \textrm{otherwise}.\\
\end{array} \right.
\end{displaymath}
We can then easily see that the rank of $M_\ell(p)$ is equal to the dimension of the vector subspace $\langle\E_{\mu,1}S_{\mu+1}, \E_{\mu,2}S_{\mu+1},\ldots,\E_{\mu,n+1}S_{\mu+1}\rangle$ in $S_\mu/S_{\mu+1}$.
\begin{lemma}\label{independence}
Let $\ell$ and $p$ be odd primes and $\mu\in\mathbb{Z}_{>0}$. For $1\leq i,j\leq 2n$, let
\begin{displaymath}
n_{ij} = \left\{ \begin{array}{ll}
1~ & \textrm{if~ $\overline{i}\cdot\overline{j}\in \{\overline{1},\overline{2},\ldots,\overline{n}\}$ in $\mathbb{Z}/\ell\mathbb{Z}$}\\
0~ & \textrm{otherwise}\\
\end{array} \right.
\end{displaymath}
and let $N_{\ell}=\big[n_{ij}\big]_{1\leq i,j\leq n}\in M_{n}(\mathbb{Z})$. 
Then the images $\E_{\mu,1}, \E_{\mu,2}, \ldots, \E_{\mu,n+1}$ are linearly independent in $S_\mu/S_{\mu+1}$ if and only if $p\nmid \det(N_\ell)$.
\end{lemma}
\begin{proof}

Let $N_{\ell}^{'}=\big[n_{ij}\big]_{\substack{1\leq i\leq n+1 \\ 1\leq j\leq 2n}}\in M_{(n+1)\times 2n}(\mathbb{Z}/p\mathbb{Z})$. It is clear that rank$\big(M_\ell(p)\big)$=rank$(N_\ell^{'})$. 
Hence the images $\E_{\mu,1}, \E_{\mu,2}, \ldots, \E_{\mu,n+1}$ are linearly independent in $S_{\mu}/S_{\mu+1}$ if and only if $N_\ell^{'}$ has rank $n+1$. 
Now, we claim that rank$(N_\ell^{'})=n+1$ if and only if $N_\ell^{'}$ induces the following row echelon form
\begin{equation}\label{matrix1}
\left[ \begin{array}{cccccccc}
~~1& & & & & & &-1 \\
 &~~1& & & & &-1 &  \\
 & &~\ddots&  & &\Ddots & &\\
 & &  &~~1 &-1 & & &\\
 & & & &~1&~1&~\cdots&~1
\end{array} \right].
\end{equation}
The ``if'' part is obvious. 
Note that if $n_{ij}=1$ (resp.~0), then $n_{i~\ell-j}=0$ (resp.~1) because there are no two automorphisms among $\V_1,\V_2, \ldots, \V_n$ which are complex conjugate of each other. 
Let $v_i$ be the $i$th row vector of $N_\ell^{'}$ for $1\leq i\leq n+1$ and let 
\begin{displaymath}
v_i^{'} = \big[v_{ij}^{'}\big]_{1\leq j\leq 2n}= \left\{ \begin{array}{ll}
2v_i-v_n-v_{n+1}~ & \textrm{for~ $1\leq i \leq n$}\\
v_n+v_{n+1}~ & \textrm{for $i=n+1$}.\\
\end{array} \right.
\end{displaymath}
Observe that $v_{n+1}^{'}=[1~~1~~\cdots~~1]$ and $2\in (\mathbb{Z}/p\mathbb{Z})^{\times}$.  
For $1\leq i \leq n$, if $v_{ij}^{'}=1$ (resp.~$-1$) then $v_{i~\ell-j}=-1$ (resp.~$1$). 
Thus we can write $v_{i}^{'}$ as a linear combination of the row vectors of the above row echelon form (\ref{matrix1}), and hence the claim is proved.
By the above claim rank$(N_\ell^{'})=n+1$ if and only if $\det\Big(\big[n_{ij}\big]_{1\leq i,j\leq n+1}\Big)\not\equiv 0\pmod{p}$. 
Since 
\begin{equation*}
\big[-n_{1j}+n_{nj}+n_{n+1~j}\big]_{1\leq j\leq n+1}=[0~~0~~\cdots~~0~~1],
\end{equation*}
we derive 
\begin{equation*}
\det\Big(\big[n_{ij}\big]_{1\leq i,j\leq n+1}\Big)=\det(N_\ell). 
\end{equation*}
This completes the proof.
\end{proof}

\begin{theorem}\label{dimension}
Let $\ell$ and $p$ be odd primes and put $n=({\ell-1})/{2}$.
Further, let $M_\ell(p)$ and $N_\ell$ be as above. 
If $p\nmid \ell h_\ell^+ n$, then for every $\mu\in\mathbb{Z}_{>0}$ we deduce
\begin{equation*}
\mathrm{Gal}(K_\mu/k_\mu)\cong (\mathbb{Z}/p\mathbb{Z})^{rank(M_\ell(p))}.
\end{equation*}
And, if $p\nmid \det(N_\ell)$, then we obtain
\begin{equation*}
\mathrm{Gal}(K_\mu/k_\mu)\cong (\mathbb{Z}/p\mathbb{Z})^{n+1}.
\end{equation*}
\end{theorem}
\begin{proof}
By (\ref{Galois}) it suffices to show that the dimension of $\langle\E_{\mu,1}S_{\mu+1}, \E_{\mu,2}S_{\mu+1},\ldots,\E_{\mu,n+1}S_{\mu+1}\rangle$ in $S_\mu/S_{\mu+1}$ is equal to the dimension of $\langle\E_{\mu,1}H_\mu, \E_{\mu,2}H_\mu,\ldots,\E_{\mu,n+1}H_\mu\rangle$ in $S_\mu/H_\mu$.
If $n=1$, then $H_\mu=S_{\mu+1}$ by Lemma \ref{unit}; hence we are done in this case. 
Thus we may assume $n\geq 2$.
It is well known that $\mathbb{Z}[\Z+\Z^{-1}]$ is the ring of integers of the maximal real subfield $\mathbb{Q}(\Z+\Z^{-1})$ of $k$ and $[\mathbb{Q}(\Z+\Z^{-1}):\mathbb{Q}]=n$ (\cite[Theorem 4] {Liang}).
Therefore, if $uS_{\mu+1}\in H_\mu/S_{\mu+1}$, then by Lemma \ref{unit} we can write 
\begin{displaymath}
u=1+2p^\mu\big(a_0+a_1(\Z+\Z^{-1})+a_2(\Z+\Z^{-1})^2+\cdots+a_{n-1}(\Z+\Z^{-1})^{n-1}\big)\in S_\mu\cap E
\end{displaymath}
for some $a_i\in \mathbb{Z}$ with $0\leq i\leq n-1$.
If $p~|~a_i$ for $1\leq i\leq n-1$, then 
\begin{eqnarray*}
N_{\mathbb{Q}(\Z+\Z^{-1})/\mathbb{Q}}(u)&\equiv& 1+2p^\mu na_0\pmod{2p^{\mu+1}}\\
&=&1.
\end{eqnarray*}
Since $p\nmid n$, we have $p~|~a_0$ and so $u\in S_{\mu+1}$.
\par
Now, we set
\begin{eqnarray*}
b&=&a_0+a_1(\Z+\Z^{-1})+a_2(\Z+\Z^{-1})^2+\cdots+a_{n-1}(\Z+\Z^{-1})^{n-1}\\
 &=&b_0+b_1\Z+b_2\Z^2+\cdots+b_n\Z^n+b_n\Z^{-n}+b_{n-1}\Z^{-(n-1)}+\cdots+b_1\Z^{-1},
\end{eqnarray*}
where $b_i\in\mathbb{Z}$ for $0\leq i\leq n$. Observe that $b_n=0$, $b_{n-1}=a_{n-1}$ and $b_{n-2}=a_{n-2}$. Consider the following matrix
\begin{displaymath}
\mathbf{M} =
\left[\begin{matrix}
1& & & & & & &-1\\
 &1& & & & &-1&\\
 & &\ddots& & &\Ddots& &\\
&&&1&-1&&&\\
&&&&1&1&\cdots&1\\
b_1-b_0& b_2-b_0 & \cdots & b_n-b_0 &b_n-b_0 &\cdots&b_2-b_0&b_1-b_0
\end{matrix}\right]
\in M_{(n+2)\times 2n}(\mathbb{Z}/p\mathbb{Z}).
\end{displaymath}
The last row of $\mathbf{M}$ is induced from the coefficients of $\Z^j$ for $1\leq j \leq 2n$ in $b$. 
So, $\mathbf{M}$ is row equivalent to
\begin{eqnarray*}
\mathbf{M} &\sim&   \left[\begin{matrix}
1& & & & & & &-1\\
 &1& & & & &-1&\\
 & &\ddots& & &\Ddots& &\\
&&&1&-1&&&\\
&&&&1&1&\cdots&1\\
\quad0\quad&\quad 0\quad &\quad \cdots\quad &\quad0\quad &2(b_n-b_0) &2(b_{n-1}-b_0)&\cdots&2(b_1-b_0)
\end{matrix}\right]\\
&\sim& \left[\begin{matrix}
1& & & & & & &-1\\
 &1& & & & &-1&\\
 & &\ddots& & &\Ddots& &\\
&&&1&-1&&&\\
&&&&1&1&\cdots&1\\
~0~& ~0~ & \cdots &~0~ &~0~ &2b_{n-1}&\cdots&2b_1
\end{matrix}\right]\quad\textrm{because $b_n=0$}.
\end{eqnarray*}
Here we claim that if $u\notin S_{\mu+1}$, then the rank of $\mathbf{M}$ is $n+2$. 
Indeed, if $p\nmid a_{n-1}$ or $p\nmid a_{n-2}$ then we are done. 
Otherwise, we get
\begin{eqnarray*}
b_{n-3}&\equiv& a_{n-3}\pmod{p}\\
b_{n-4}&\equiv& a_{n-4}\pmod{p}.
\end{eqnarray*}
Hence by induction we ensure that the rank of $\mathbf{M}$ is $n+1$ if and only if $p~|~a_i$ for all $1\leq i \leq n-1$. 
Since $u\notin S_{\mu+1}$, we obtain $p\nmid a_i$ for some $1\leq i \leq n-1$, and the claim is proved. 
In the proof of Lemma \ref{independence} we already showed that every row vector of $M_\ell(p)$ can be written as a linear combination of the row vectors of the matrix (\ref{matrix1}). 
Therefore, if $u\notin S_{\mu+1}$ then by the above claim for each $1\leq i \leq n+1$ the images of $\E_{\mu,i}$ and $u$ in $S_\mu/S_{\mu+1}$ are linearly independent, as desired.
Furthermore if $p\nmid\det(N_\ell)$, then by Lemma \ref{independence} we achieve rank$\big(M_\ell(p)\big)=n+1$.
\end{proof}

\begin{corollary}\label{unit2}
Suppose $p\nmid \ell h_\ell^+ n$.
Then $K_\mu$ becomes the ray class field $k_{\mu+1}$ for all $\mu\in\mathbb{Z}_{>0}$ if and only if $\dim_{\mathbb{Z}/p\mathbb{Z}}(H_1/S_2)=n-1$ and $p\nmid \det(N_\ell)$. 
\end{corollary}
\begin{proof}
Suppose that $\dim_{\mathbb{Z}/p\mathbb{Z}}(H_1/S_2)=n-1$ and $p\nmid \det(N_\ell)$.
We claim that $\dim_{\mathbb{Z}/p\mathbb{Z}}(H_\mu/S_{\mu+1})=n-1$ for all $\mu\in\mathbb{Z}_{>0}$.
Indeed, let $\varepsilon_1, \varepsilon_2, \ldots, \varepsilon_{n-1}$ be elements of $S_1\cap E$ whose images form a basis for $H_1/S_2$. 
Since $\varepsilon_i^p \in H_2-S_3$ for all $i$, the images $\varepsilon_1^p, \varepsilon_2^p, \ldots, \varepsilon_{n-1}^p$ turn out to be a basis for $H_2/S_3$. 
And by induction the claim is proved.
By Theorem \ref{dimension} and the assumptions, we deduce
\begin{eqnarray*}
|\mathrm{Gal}(K_{\mu}/k_\mu)|&=&|\widetilde{\V_\mu^*}(S_\mu/H_\mu)|=p^{n+1},\\
|\mathrm{Gal}(k_{\mu+1}/k_\mu)|&=&|S_\mu/H_\mu|=\frac{|S_\mu/S_{\mu+1}|}{|H_\mu/S_{\mu+1}|}=p^{n+1}.
\end{eqnarray*}
Therefore $K_\mu=k_{\mu+1}$ because $K_\mu\subset k_{\mu+1}$.
\par
Conversely, suppose that $K_\mu=k_{\mu+1}$ for all $\mu\in\mathbb{Z}_{>0}$.
It follows from the proof of Lemma \ref{unit} that $|H_\mu/S_{\mu+1}|\leq p^{n-1}$, and so $|\mathrm{Gal}(k_{\mu+1}/k_\mu)|\geq p^{n+1}$.
On the other hand, $|\mathrm{Gal}(K_{\mu}/k_\mu)|\leq p^{n+1}$ by the formula (\ref{Galois}).
Since $K_\mu=k_{\mu+1}$, we have $\dim_{\mathbb{Z}/p\mathbb{Z}}(H_1/S_2)=n-1$ and $|\mathrm{Gal}(K_{\mu}/k_\mu)|= p^{n+1}$; hence $p\nmid \det(N_\ell)$ by Lemma \ref{independence} and Theorem \ref{dimension}.
\end{proof}

\begin{remark}
\begin{itemize}
\item[(i)] 
We are able to show that $\det(N_\ell)\neq 0$ for $\ell\leq 10000$ with an aid of Maple software, from which we conjecture that $\det(N_\ell)\neq 0$ holds for all odd primes $\ell$. 
\item[(ii)] Let $\xi_a=\Z_{2\ell}^{{1-a}} ({1-\Z^a})/({1-\Z})\in\mathbb{R}$ for $1 < a < {\ell}/{2}$. When $\ell=5$, we obtain
\begin{eqnarray*}
\xi_2^{48}&\equiv& 1 \pmod{14}\\
&\not\equiv&1 \pmod{98}.
\end{eqnarray*}
Thus $\dim_{\mathbb{Z}/7\mathbb{Z}}(H_1/S_2)=1$, and Corollary \ref{unit2} is true for $p=7$. 
In a similar way, we can show that $\dim_{\mathbb{Z}/p\mathbb{Z}}(H_1/S_2)=1$ holds for all odd primes $p\leq 101$ except $p=3$.
When $\ell=7$, we see $\dim_{\mathbb{Z}/p\mathbb{Z}}(H_1/S_2)=2$ for all odd primes $p\leq 101$.
So, we also conjecture that for each odd prime $\ell$, $\dim_{\mathbb{Z}/p\mathbb{Z}}(H_1/S_2)=n-1$ holds for almost all odd primes $p$.
If the above two conjectures are true, then we have the isomorphism
\begin{equation*}
\mathrm{Gal}(k_{\mu+1}/k_\mu)\cong (\mathbb{Z}/p\mathbb{Z})^{n+1}
\end{equation*}
for each odd prime $\ell$ and almost all odd primes $p$.
\item[(iii)] The following table lists the generators of $H_1/S_2$ for $\ell=5,7$.

\begin{tabular}{|c|c|c|c|c|c|c|}
\cline{1-3} \cline{5-7}
\multicolumn{1}{|c|}{$p$} & \multicolumn{2}{c|}{generators of $H_1/S_2$} && \multicolumn{1}{|c|}{$p$} & \multicolumn{2}{c|}{generators of $H_1/S_2$}\\
\cline{2-3} \cline{6-7}
& $\ell$=5 & $\ell$=7 &&& $\ell$=5 & $\ell$=7\\
\cline{1-3} \cline{5-7}
3 & 1 & $\X_2^{182}, \X_3^{182}$ & &47& $\X_2^{2208}$& $\X_2^{726754}, \X_3^{726754}$\\
5 & $\X_2^{60}$  & $\X_2^{868}, \X_3^{868}$ & &53& $\X_2^{1404}$& $\X_2^{148876}, \X_3^{148876}$\\
7 & $\X_2^{48}$  & $\X_2^{42}, \X_3^{42}$ & &59& $\X_2^{174}$ & $\X_2^{1437646}, \X_3^{1437646}$\\
11& $\X_2^{30}$  & $\X_2^{1330}, \X_3^{1330}$ && 61& $\X_2^{60}$  & $\X_2^{40740}, \X_3^{40740}$\\
13& $\X_2^{84}$  & $\X_2^{84}, \X_3^{84}$ & &67& $\X_2^{4488}$& $\X_2^{100254}, \X_3^{100254}$\\
17& $\X_2^{72}$  & $\X_2^{17192}, \X_3^{34384}$ && 71& $\X_2^{210}$ & $\X_2^{70}, \X_3^{70}$\\
19& $\X_2^{18}$  & $\X_2^{16002}, \X_3^{16002}$ && 73& $\X_2^{1332}$& $\X_2^{226926}, \X_3^{453852}$\\
23& $\X_2^{528}$ & $\X_2^{12166}, \X_3^{12166}$ && 79& $\X_2^{78}$  & $\X_2^{164346}, \X_3^{164346}$\\ 
29& $\X_2^{848}$ & $\X_2^{28}, \X_3^{28}$ && 83& $\X_2^{6888}$& $\X_2^{574}, \X_3^{574}$\\
31& $\X_2^{30}$  & $\X_2^{69510}, \X_3^{69510}$ && 89& $\X_2^{132}$ & $\X_2^{1233694}, \X_3^{4934776}$\\
37& $\X_2^{684}$ & $\X_2^{16884}, \X_3^{16884}$ && 97& $\X_2^{2352}$ & $\X_2^{672}, \X_3^{672}$\\
41& $\X_2^{120}$ & $\X_2^{280}, \X_3^{280}$ && 101& $\X_2^{300}$ & $\X_2^{7212100}, \X_3^{7212100}$\\
\cline{5-7} 
43& $\X_2^{1848}$& $\X_2^{42}, \X_3^{42}$\\
\cline{1-3}
\end{tabular}
\end{itemize}
\end{remark}

\section{Theta functions}\label{Theta functions}
In this section we shall provide necessary fundamental transformation formulas of theta functions and describe the action of $G_{\mathbb{A}+}$ on the quotient of two theta-constants.
\par

Let $n$ be a positive integer, $u\in\mathbb{C}^n$, $z\in\mathbb{H}_n$ and $r,s\in \mathbb{R}^n$. 
We define a (classical) \textit{theta function} by
\begin{equation*}
\T(u,z;r,s)=\sum_{x\in \mathbb{Z}^n}e\left(\frac{1}{2}\cdot{^t}(x+r)z(x+r)+{^t}(x+r)(u+s)\right).
\end{equation*}

\begin{proposition}\label{translation}
Let $r,s\in\mathbb{R}^n$ and $a,b\in\mathbb{Z}^n$.
\begin{itemize}
\item[\textup{(i)}]  $\T(-u,z;-r,-s)=\T(u,z;r,s)$.
\item[\textup{(ii)}] $\T(u,z;r+a,s+b)=e({^t}rb)\T(u,z;r,s)$.
\end{itemize}
\end{proposition}
\begin{proof}
\cite[p.676 (13)]{Shimura1}.
\end{proof}

For a square matrix $S$, by $\{S\}$ we mean the column vector whose components are the diagonal entries of $S$.
\begin{proposition}\label{transformation}
For every $\gamma= \left[\begin{matrix}
A&B\\
C&D
\end{matrix}\right]\in\G(1)$ such that $\{{^t}AC\}, \{{^t}BD\}\in 2\mathbb{Z}^n$, 
we get the transformation formula
\begin{equation*}
\begin{array}{ccl}
\T\big({^t}(Cz+D)^{-1}u,\gamma(z);r,s\big)&=&
\displaystyle\lambda_\gamma e\Big(\frac{{^t}rs-{^t}r's'}{2}\Big)\det(Cz+D)^{1/2}e\Big(\frac{1}{2}\cdot{^t}u(Cz+D)^{-1}cu  \Big)\times\vspace{0.1cm}\\
&&\T(u,z;r',s'),
\end{array}
\end{equation*}
where $\lambda_\gamma$ is a constant of absolute value $1$ depending only on $\gamma$ and the choice of the branch of 
$\det(Cz+D)^{1/2}$, and
\begin{equation*}
\left[\begin{matrix}
r'\\
s'
\end{matrix}\right]
={^t}\gamma
\left[\begin{matrix}
r\\
s
\end{matrix}\right].
\end{equation*}
In particular, $\lambda_\gamma^4=1$ for $\gamma\in\G(2)$.
\end{proposition}
\begin{proof}
\cite[Proposition 1.3 and Proposition 1.4]{Shimura1}.
\end{proof}
Here, the functions $\T(0,z;r,s)$ are called \textit{theta-constants}, and these are holomorphic on $\mathbb{H}_n$ as functions in $z$ (\cite[Proposition 1.6]{Shimura1}).
\begin{proposition}\label{theta-constant}
Suppose that $r,s$ belong to in $\mathbb{Q}^n$. 
Then the theta constant $\T(0,z;r,s)$ represents the zero function if and only if $r,s\in({1}/{2})\mathbb{Z}^n$ and $e(2\cdot{^t}rs)=-1$.
\end{proposition}
\begin{proof}
\cite[Theorem 2]{Igusa}.
\end{proof}
Let
\begin{equation*}
\Phi_{(r,s)}(z)=\frac{\T(0,z;r,s)}{\T(0,z;0,0)}.
\end{equation*}
Note that the poles of $\Phi_{(r,s)}(z)$ are exactly the zeros of $\T(0,z;0,0)=\sum_{x\in \mathbb{Z}^n}e\Big(\frac{1}{2}~{^t}xzx\Big)$.
When $n=1$, $\T(0,z;0,0)$ has no zero on $\mathbb{H}_1$ by the Jacobi's triple product identity \cite[Theorem 14.6]{Apostol}.
\begin{lemma}\label{theta}
For $r,s\in\mathbb{R}^n$ and $a,b\in\mathbb{Z}^n$, we achieve that
\begin{itemize}
\item[\textup{(i)}] $\Phi_{(-r,-s)}(z)=\Phi_{(r,s)}(z)$,
\item[\textup{(ii)}] $\Phi_{(r+a,s+b)}(z)=e({^t}rb)\Phi_{(r,s)}(z)$,
\item[\textup{(iii)}] If
$\gamma= \left[\begin{matrix}
A&B\\
C&D
\end{matrix}\right]\in\G(1)$ such that $\{{^t}AC\}, \{{^t}BD\}\in 2\mathbb{Z}^n$, then we obtain
\begin{equation*}
\Phi_{(r,s)}\big(\gamma(z)\big)=e\Big(\frac{{^t}rs-{^t}r's'}{2}\Big)\Phi_{(r',s')}(z),
\end{equation*}
where
\begin{equation*}
\left[\begin{matrix}
r'\\
s'
\end{matrix}\right]
={^t}\gamma
\left[\begin{matrix}
r\\
s
\end{matrix}\right].
\end{equation*}

\end{itemize}
\end{lemma}
\begin{proof}
It is immediate from Proposition \ref{translation} and Proposition \ref{transformation}.
\end{proof}

\begin{proposition}\label{phi-action}
Let $m$ be a positive integer and let $r,s\in ({1}/{m})\mathbb{Z}^n$. 
Then $\Phi_{(r,s)}(z)$ belongs to $\mathcal{F}_{2m^2}$.
Moreover, if $x$ is an element of $\Delta$ such that
\begin{equation*}
x_q\equiv \left[\begin{matrix}
1_n&0\\
0&t1_n
\end{matrix}\right]
\pmod{2 m^2 M_{2n}(\mathbb{Z}_q)}
\end{equation*}
for all rational primes $q$ and a positive integer $t$, then
\begin{equation*}
\Phi_{(r,s)}(z)^x=\Phi_{(r,ts)}(z).
\end{equation*}
\end{proposition}
\begin{proof}
\cite[Proposition 1.7]{Shimura1}.
\end{proof}

\begin{corollary}\label{action}
For $m\in \mathbb{Z}_{>0}$ and $r,s\in({1}/{m})\mathbb{Z}^n$, let
\begin{equation*}
y=\beta\left[\begin{matrix}1_n&0\\0&x\cdot 1_n\end{matrix}\right]\alpha\in G_{\mathbb{A}+}
\end{equation*}
with $\beta\in R_{2m^2}$, $x\in\displaystyle\prod_q\mathbb{Z}_q^\times$, and $\alpha\in G_{\mathbb{Q}+}$. Then 
\begin{equation*}
{(\Phi_{(r,s)})}^y(z)=\Phi_{(r,ts)}\big(\alpha(z)\big),
\end{equation*}
where $t$ is a positive integer such that $t\equiv x_q\pmod{2m^2\mathbb{Z}_q}$ for all rational primes $q$.
\end{corollary}
\begin{proof}
This can be proved by Proposition \ref{Siegel-action} and Proposition \ref{phi-action}.
\end{proof}

\section{Construction of class fields}\label{Construction of class fields}
We use the same notations as in Section \ref{Class fields over cyclotomic fields}.
Let $k=\mathbb{Q}(\Z)$ with $\Z=\Z_\ell$ and $n=({\ell-1})/{2}$ so that $2n=[k:\mathbb{Q}]$.
Let $v:k\rightarrow \mathbb{C}^n$ be the map given by $v(x)=
\left[\begin{matrix}
x^{\V_1}\\
\vdots\\
x^{\V_n}
\end{matrix}\right]$
, $L=v(\mathcal{O}_k)$ be a lattice in $\mathbb{C}^n$ and $\rho=({\Z-\Z^{-1}})/{\ell}\in k$. Then $\rho$ satisfies the conditions (i)$\sim$(iv) in \S\ref{Shimura's reciprocity law}.
And, we have an $\mathbb{R}$-bilinear form $E:\mathbb{C}^n\times\mathbb{C}^n\rightarrow\mathbb{R}$ defined by
\begin{equation*}
E(z,w)=\sum_{i=1}^{n} \rho^{\V_i}(z_i\overline{w_i}-\overline{z_i}w_i)~~\textrm{ for $z=\left[\begin{matrix}
z_1\\
\vdots\\
z_n
\end{matrix}\right]$,  $w=\left[\begin{matrix}
w_1\\
\vdots\\
w_n
\end{matrix}\right]$},
\end{equation*}
which induces a non-degenerate Riemann form on $\mathbb{C}^n/L$.
Let
\begin{displaymath}
e_i = \left\{ \begin{array}{ll}
\Z^{2i}~ & \textrm{for $1\leq i\leq n$}\\
\displaystyle\sum_{j=1}^{i-n}\Z^{2j-1}~ & \textrm{for $n+1\leq i\leq 2n$}.
\end{array} \right.
\end{displaymath}
Since $\{e_1,e_2,\ldots,e_{2n} \}$ is a free $\mathbb{Z}$-basis of $\mathcal{O}_k$, $\{v(e_1),v(e_2),\ldots,v(e_{2n})\}$ is a free $\mathbb{Z}$-basis of the lattice $L$, and we get
\begin{displaymath}
\Big[E\big(v(e_i),v(e_j)\big)\Big]_{1\leq i,j\leq 2n}=J.
\end{displaymath}
Now, let
\begin{displaymath}
\Omega=\big[v(e_1)~~v(e_2)~~\cdots~~v(e_{2n})\big]\in M_{n\times 2n}(\mathbb{C}).
\end{displaymath}
Then $\Omega$ satisfies
\begin{eqnarray*}
L&=&\Big\{\Omega\left[\begin{matrix} a\\b \end{matrix}\right]~\big|~a,b\in\mathbb{Z}^n \Big\},\\
E(\Omega x, \Omega y)&=&{^t}xJy ~~~\textrm{for $x,y\in\mathbb{R}^{2n}$}
\end{eqnarray*}
because $E$ is $\mathbb{R}$-bilinear. 
Thus $\delta=1$ and $\epsilon=1_n$ in \S\ref{Shimura's reciprocity law}.
Write $\Omega=[\Omega_1~~\Omega_2]$ with $\Omega_1,\Omega_2\in M_n(\mathbb{C})$ and put $z_\ell=\Omega_2^{-1}\Omega_1\in\mathbb{H}_n$. 
We define a ring monomorphism $h$ of $k_{\mathbb{A}}$ into $M_{2n}(\mathbb{Q}_\mathbb{A})$ as in \S \ref{Shimura's reciprocity law}.
Then $z_\ell$ is the CM-point of $\mathbb{H}_n$ induced from $h$ corresponding to the polarized abelian variety $(\mathbb{C}^n/L, E)$.
\par
Let $p$ be an odd prime and $r,s\in\mathbb{Q}^n$. 
We denote by $\mathbf{h}$ the set of all non-archimedean primes of $k$ and by $\mathbf{a}$ the set of all archimedean primes of $k$. 
For given $\O\in\mathcal{O}_k$ prime to $2p$, we set
\begin{equation*}
\widetilde{\O}=\prod_{\substack{v\in\mathbf{h} \\v\,|\,2p }}(\O^{-1})_v\times\prod_{\substack{v\in\mathbf{h} \\v\,\nmid\, 2p }}1_v \times\prod_{v\in\mathbf{a}}1_v \in k_\mathbb{A}^\times.
\end{equation*}
Here $x_v$ is the $v$-component of $x\in k_\mathbb{A}^\times$.
If $\Phi_{(r,s)}$ is finite at $z_\ell$ then by Proposition \ref{reciprocity} and \cite[Chapter 8 $\S$4]{Lang}, we have
\begin{equation*}
\Phi_{(r,s)}(z_\ell)^{\big(\frac{k'/k}{(\O)}\big)}=\Phi_{(r,s)}(z_\ell)^{[\widetilde{\O},k]}={(\Phi_{(r,s)})}^{h(\V^*(\widetilde{\O}^{-1}))}(z_\ell)~~\textrm{for $\O\in\mathcal{O}_k$},
\end{equation*}
where $k'$ is a finite abelian extension of $k$ containing $\Phi_{(r,s)}(z_\ell)$.

\begin{lemma}\label{phi-class}
Let $p$ be an odd prime, $\mu\in\mathbb{Z}_{>0}$, $r,s\in({1}/{p^\mu})\mathbb{Z}^n$ and $z_\ell$ be as above. 
Assume that $z_\ell$ is not a zero of $\T(0,z;0,0)$.
If $p\nmid \ell h_\ell^+ n$, then $\Phi_{(r,s)}(z_\ell)^{p^\alpha}\in K_{2\mu-1-\alpha}$ for $\alpha=0,1,\ldots, \mu$.
\end{lemma}
\begin{proof}
By Proposition \ref{phi-action} and \cite[p.682]{Shimura1}, $\Phi_{(r,s)}(z)$ belongs to $\mathcal{F}_{2p^{2\mu}}$ as a function, so it is $R_{2p^{2\mu}}$-invariant.
Let $\O H_{2\mu-1-\alpha}$ belong to $\mathrm{ker}(\widetilde{\V^*_{2\mu-1-\alpha}})$ such that $\V^*(\O) \in H_{2\mu-1-\alpha}$. 
Since $p\nmid \ell h_\ell^+ n$, it follows from the proof of Theorem \ref{dimension} that 
$(\V^*(S_{2\mu-1-\alpha})\cap H_{2\mu-1-\alpha})/S_{2\mu-\alpha} =\{0\}$.
Hence $\V^*(\O) \in S_{2\mu-\alpha}$.
Write $\V^*(\O)=1+2p^{2\mu-\alpha}\O_0$ with $\O_0\in\mathcal{O}_k$.
Then it suffices to show $\big(\Phi_{(r,s)}(z_\ell)^{p^\alpha}\big)^{\big(\frac{k'/k}{(\O)}\big)}=\Phi_{(r,s)}(z_\ell)^{p^\alpha}$, where $k'$ is a finite abelian extension of $k$ containing $\Phi_{(r,s)}(z_\ell)$.
By the strong approximation theorem for $\mathrm{Sp}(n)$ there exists a matrix $\beta\in\G(1)$ such that
\begin{equation*}
h\big(\V^*(\O)\big)\equiv
\left[\begin{matrix}
1_n&0\\
0&v\cdot1_n
\end{matrix}\right]
\beta\pmod{2p^{2\mu}},
\end{equation*}
where $v:=\nu\big(h(\V^*(\O))\big)=N_{k/\mathbb{Q}}(\O)=\V^*(\O)\cdot\overline{\V^*(\O)}= 1+2p^{2\mu-\alpha}v_0$ for some $v_0\in\mathbb{Z}$.
In fact, $\beta$ belongs to $\G(2)$ owing to the facts 
$h\big(\V^*(\O)\big)\equiv 1_{2n}\pmod{2}$ and $v\equiv 1\pmod{2}$.
Thus for all rational primes $q$ we obtain
\begin{equation*}
h\big(\V^*(\widetilde{\O}^{-1})\big)_q\equiv 
\left[\begin{matrix}
1_n&0\\
0&v\cdot 1_n
\end{matrix}\right]
\beta \pmod{2p^{2\mu} M_{2n}(\mathbb{Z}_q)}.
\end{equation*}
And, by Lemma \ref{theta} and Corollary \ref{action} we get that
\begin{eqnarray*}
{(\Phi_{(r,s)})}^{h(\V^*(\widetilde{\O}^{-1}))}(z_\ell)&=&\Phi_{(r,vs)}(\beta z_\ell)\\
&=&e\Big(\frac{{^t}rvs-{^t}r's'}{2}\Big)\Phi_{(r',s')}(z_\ell),
\end{eqnarray*}
where 
\begin{eqnarray*}
\left[\begin{matrix}
r'\\s'
\end{matrix}\right]
={^t}\beta\left[\begin{matrix}r\\vs\end{matrix}\right]
&\equiv& {^t}h\big(\V^*(\O)\big)\left[\begin{matrix}r\\s\end{matrix}\right]\pmod{2p^{2\mu}}\\
&\equiv&\left[\begin{matrix}r\\s\end{matrix}\right]+2p^{2\mu-\alpha}\cdot{^t}h(\O_0)\left[\begin{matrix}r\\s\end{matrix}\right]
\pmod{2p^{2\mu}}.
\end{eqnarray*}
Let $a,b\in 2p^{\mu-\alpha}\mathbb{Z}^n\subset 2\mathbb{Z}^n$ such that $\left[\begin{matrix}
a\\
b
\end{matrix}\right]
=2p^{2\mu-\alpha}\cdot{^t}h(\O_0)\left[\begin{matrix}r\\s\end{matrix}\right]$. 
Then we derive by Lemma \ref{theta} and Proposition \ref{reciprocity} that
\begin{equation*}
\begin{array}{lll}
\big(\Phi_{(r,s)}(z_\ell)^{p^\alpha}\big)^{\big(\frac{k'/k}{(\O)}\big)}&=&\big({\Phi_{(r,s)}}^{h(\V^*(\widetilde{\O}^{-1}))}(z_\ell)\big)^{p^\alpha}\vspace{0.1cm}\\
&=&\displaystyle\Big(e\Big(\frac{{^t}rvs-{^t}(r+a)(s+b)}{2}\Big)e({^t}rb)\Phi_{(r,s)}(z_\ell)\Big)^{p^\alpha}\vspace{0.1cm}\\
&=&\displaystyle e(p^{2\mu}v_0\cdot{^t}rs)e\Big(p^\alpha\cdot\frac{{^t}rb-{^t}as}{2}\Big)\Phi_{(r,s)}(z_\ell)^{p^\alpha}\vspace{0.1cm}\\
&=&\Phi_{(r,s)}(z_\ell)^{p^\alpha}.
\end{array}
\end{equation*}
\end{proof}

Let $\mu\in\mathbb{Z}_{>0}$. Assume that $r,s\in({1}/{p^\mu})\mathbb{Z}^n$ and $z_\ell$ is not a zero of $\T(0,z;0,0)$.
Consider the matrices $h\big(\V^*(\O_{2\mu-1-\alpha,j})\big)$ for $1\leq j \leq n+1$ and $0\leq\alpha\leq \mu-1$. 
Then we have
\begin{eqnarray*}
h\big(\V^*(\O_{2\mu-1-\alpha,j})\big)&=&h\big(1+2p^{2\mu-1-\alpha}\V^+(\Z^j)+2p^{2\mu-\alpha}\O_0\big)\\ &=&1_{2n}+2p^{2\mu-1-\alpha}h\big(\V^+(\Z^j)\big)+2p^{2\mu-\alpha}h(\O_0)
\end{eqnarray*}
for some $\O_0\in\mathcal{O}_k$.
Also, we can deduce without difficulty
\begin{equation*}
h(\Z)=
\left[ \begin{array}{ccccc|ccccc}
0&0 &\cdots &\cdots &0 &-1&1& & & \\
\vdots&\vdots& & &\vdots&&-1 &1 & &  \\
\vdots&\vdots& & &\vdots& &&\ddots &\ddots &\\
0&0 &\cdots &\cdots &0& && &-1 &1\\
-1&-1&\cdots&\cdots&-1& & & & &-1\\
\hline
1 & & &  &&0 &0&\cdots &\cdots &0\\
1&1& &  &&\vdots&\vdots&& &\vdots \\
1 &1&1&& &\vdots&\vdots& &&\vdots \\
\vdots&\vdots& &\ddots & &\vdots&\vdots& & &\vdots \\
1&1&1&\cdots&1&0&0&\cdots&\cdots&0
\end{array} \right].
\end{equation*}
Now, again by the strong approximation theorem for $\mathrm{Sp}(n)$ there exists a matrix $\beta_{2\mu-1-\alpha,j}$ in $\G(1)$ such that
\begin{equation*}
h\big(\V^*(\O_{2\mu-1-\alpha,j})\big)\equiv
\left[\begin{matrix}
1_n&0\\
0&v_{2\mu-1-\alpha,j}\cdot 1_n
\end{matrix}\right]
\beta_{2\mu-1-\alpha,j}\pmod{2p^{2\mu}},
\end{equation*}
where $v_{2\mu-1-\alpha,j}:=\nu\big(h(\V^*(\O_{2\mu-1-\alpha,j}))\big)=N_{k/\mathbb{Q}}(\O_{2\mu-1-\alpha,j})= 1-2p^{2\mu-1-\alpha}+2p^{2\mu-\alpha}v_j$ for some $v_j\in\mathbb{Z}$.
As a matter of fact, $\beta_{2\mu-1-\alpha,j}$ belongs to $\G(2)$ due to the facts 
$h\big(\V^*(\O_{2\mu-1-\alpha,j})\big)\equiv 1_{2n}\pmod{2}$ and $v_{2\mu-1-\alpha,j}\equiv 1\pmod{2}$.
And for all rational primes $q$ we obtain
\begin{equation*}
h\big(\V^*(\widetilde{\O}_{2\mu-1-\alpha,j}^{-1})\big)_q\equiv 
\left[\begin{matrix}
1_n&0\\
0&v_{2\mu-1-\alpha,j}\cdot 1_n
\end{matrix}\right]
\beta_{2\mu-1-\alpha,j} \pmod{2p^{2\mu} M_{2n}(\mathbb{Z}_q)}.
\end{equation*}
By Lemma \ref{theta} and Corollary \ref{action} we obtain
\begin{eqnarray*}
{\Phi_{(r,s)}}^{h(\V^*(\widetilde{\O}_{2\mu-1-\alpha,j}^{-1}))}(z_\ell)&=&\Phi_{(r,v_{2\mu-1-\alpha,j}s)}(\beta_{2\mu-1-\alpha,j}(z_\ell))\\
&=&e\Big(\frac{{^t}rv_{2\mu-1-\alpha,j}s-{^t}r's'}{2}\Big)\Phi_{(r',s')}(z_\ell),
\end{eqnarray*}
where 
\begin{eqnarray*}
\left[\begin{matrix}
r'\\s'
\end{matrix}\right]
&=&{^t}\beta_{2\mu-1-\alpha,j}\left[\begin{matrix}r\\v_{2\mu-1-\alpha,j}s\end{matrix}\right]\\
&\equiv& {^t}h\big(\V^*(\O_{2\mu-1-\alpha,j})\big)\left[\begin{matrix}r\\s\end{matrix}\right]\pmod{2p^{2\mu}}\\
&\equiv&\left[\begin{matrix}r\\s\end{matrix}\right]+2p^{2\mu-1-\alpha}\cdot{^t}h\big(\V^+(\Z^j)\big)\left[\begin{matrix}r\\s\end{matrix}\right]
+2p^{2\mu-\alpha}\cdot{^t}h(\O_0)\left[\begin{matrix}r\\s\end{matrix}\right]
\pmod{2p^{2\mu}}.
\end{eqnarray*}
For each $j$, let $a_{2\mu-1-\alpha,j},b_{2\mu-1-\alpha,j}\in 2p^{\mu-1-\alpha}\mathbb{Z}^n\subset 2\mathbb{Z}^n$ such that 
\begin{equation*}
\left[\begin{matrix}
a_{2\mu-1-\alpha,j}\\
b_{2\mu-1-\alpha,j}
\end{matrix}\right]
=2p^{2\mu-1-\alpha}\cdot{^t}h\big(\V^+(\Z^j)\big)\left[\begin{matrix}r\\s\end{matrix}\right].
\end{equation*} 
And, let $c, d\in 2p^{\mu-\alpha}\mathbb{Z}^n\subset 2\mathbb{Z}^n$ for which 
$\left[\begin{matrix}
c\\
d
\end{matrix}\right]
=2p^{2\mu-\alpha}\cdot{^t}h(\O_0)\left[\begin{matrix}r\\s\end{matrix}\right]$.
Then by Lemma \ref{theta} and Proposition \ref{reciprocity} we derive that
\begin{equation}\label{formula}
\begin{array}{ccl}
\big(\Phi_{(r,s)}(z_\ell)^{p^\alpha}\big)^{\big(\frac{K_{2\mu-1-\alpha}/k}{(\O_{2\mu-1-\alpha,j})}\big)}
&=&\big({\Phi_{(r,s)}}^{h(\V^*(\widetilde{\O}_{2\mu-1-\alpha,j}^{-1}))}(z_\ell)\big)^{p^\alpha}\vspace{0.1cm}\\
&=&\displaystyle e\big(p^\alpha\frac{{^t}rv_{2\mu-1-\alpha,j}s-{^t}(r+a_{2\mu-1-\alpha,j}+c)(s+b_{2\mu-1-\alpha,j}+d)}{2}\big)\times
\vspace{0.1cm}\\
&&e\big(p^\alpha\cdot{^t}r(b_{2\mu-1-\alpha,j}+d)\big)\Phi_{(r,s)}(z_\ell)^{p^\alpha}\vspace{0.1cm}\\
&=&\displaystyle e(-p^{2\mu-1}\cdot{^t}rs)e\Big(p^\alpha\frac{{^t}rb_{2\mu-1-\alpha,j}-{^t}a_{2\mu-1-\alpha,j}s}{2}\Big)\Phi_{(r,s)}(z_\ell)^{p^\alpha}\vspace{0.1cm}\\
&=&\underbrace{e(-p^{2\mu-1}\cdot{^t}rs)e\Big(\frac{{^t}rb_{2\mu-1,j}-{^t}a_{2\mu-1,j}s}{2}\Big)}_{\textrm{$p$-th root of unity}}\Phi_{(r,s)}(z_\ell)^{p^\alpha}.
\end{array}
\end{equation}
\par

For given $r_i, s_i\in \mathbb{Z}^n$ with $1\leq i \leq n+1$, let $\mathbf{r}=[r_i]_{1\leq i \leq n+1}$ and $\mathbf{s}=[s_i]_{1\leq i \leq n+1}$.
And, we set  $\Phi_{[\mathbf{r},\mathbf{s};\mu,i]}(z)_p=\Phi_{(r_i/p^\mu,s_i/p^\mu)}(z)$ for each $\mu\in\mathbb{Z}_{>0}$.
\begin{lemma}\label{depend}
Let $\mathbf{r}=[r_i]_{1\leq i \leq n+1}$ and $\mathbf{s}=[s_i]_{1\leq i \leq n+1}$ for given $r_i, s_i\in\mathbb{Z}^n$.
Then for each $1\leq i,j\leq n+1$, there exists an integer $c_{ij}$ such that 
\begin{equation*}
\Phi_{[\mathbf{r},\mathbf{s};1,i]}(z_\ell)_p^{\big(\frac{K_1/k}{(\O_{1,j})}\big)}=\Z_p^{c_{ij}}\Phi_{[\mathbf{r},\mathbf{s};1,i]}(z_\ell)_p
\end{equation*}
for all odd primes $p$.
\end{lemma}
\begin{proof}
Substituting $r=r_i/p$, $s=s_i/p$ into the formula (\ref{formula}) we get
\begin{equation*}
\begin{array}{lll}
\Phi_{[\mathbf{r},\mathbf{s};1,i]}(z_\ell)_p^{\big(\frac{K_1/k}{(\O_{1,j})}\big)}&=&\displaystyle e\Big(-\frac{{^t}r_i s_i}{p}\Big)e\Big(\frac{{^t}r_i b_{1,j}-{^t}a_{1,j}s_i}{2p}\Big)\Phi_{[\mathbf{r},\mathbf{s};1,i]}(z_\ell)_p\vspace{0.1cm}\\
&=&\displaystyle\Z_p^{-{^t}r_i s_i+({^t}r_i b_{1,j}-{^t}a_{1,j}s_i)/2}\Phi_{[\mathbf{r},\mathbf{s};1,i]}(z_\ell)_p,
\end{array}
\end{equation*}
where $a_{1,j},b_{1,j}\in 2\mathbb{Z}^n$ such that $\left[\begin{matrix}
a_{1,j}\\
b_{1,j}
\end{matrix}\right]
=2\cdot{^t}h\big(\V^+(\Z^j)\big)\left[\begin{matrix}r_i\\s_i\end{matrix}\right]$.
Hence 
\begin{equation*}
c_{ij}:=-{^t}r_i s_i+\frac{1}{2}({^t}r_i b_{1,j}-{^t}a_{1,j}s_i)
\end{equation*}
is an integer which does not depend on $p$.
\end{proof}
We put $A_\ell(\mathbf{r},\mathbf{s})=[c_{ij}]_{1\leq i,j\leq n+1}\in M_{n+1}(\mathbb{Z})$ where $c_{ij}$ is an integer satisfying 
\begin{equation*}
\Phi_{[\mathbf{r},\mathbf{s};1,i]}(z_\ell)_p^{\big(\frac{K_1/k}{(\O_{1,j})}\big)}=\Z_p^{c_{ij}}\Phi_{[\mathbf{r},\mathbf{s};1,i]}(z_\ell)_p
\end{equation*}
for all odd primes $p$.
Note that $A_\ell(\mathbf{r},\mathbf{s})$ does not depend on $p$ by lemma \ref{depend}.
\par
Now, in order to construct a primitive generator of $K_\mu$ over $k_\mu$, we need the following lemma.

\begin{lemma}\label{primitive}
Let $L$ be an abelian extension of a number field $F$.
Suppose that $L=F(\alpha, \beta)$ for some $\alpha, \beta\in L$.
Let $a, b$ be any nonzero elements of $F$ and $\nu=[L : F(\alpha)]$.
Then we have
\begin{equation*}
L=F\big(a\alpha+b(\nu\beta-\mathrm{Tr}_{L/F(\alpha)}(\beta))\big).
\end{equation*}
\end{lemma}
\begin{proof}
Let $\varepsilon=a\alpha+b(\nu\beta-\mathrm{Tr}_{L/F(\alpha)}(\beta))\in L$.
Then we have
\begin{equation}\label{trace}
\mathrm{Tr}_{L/F(\alpha)}(\varepsilon)=a\alpha\mathrm{Tr}_{L/F(\alpha)}(1)+b\nu\mathrm{Tr}_{L/F(\alpha)}(\beta)-b\mathrm{Tr}_{L/F(\alpha)}(\beta)\mathrm{Tr}_{L/F(\alpha)}(1)=a\alpha\nu.
\end{equation}
Since $L$ is an abelian extension of $F$, so is $F(\varepsilon)$ by Galois theory.
Thus $\varepsilon^\sigma\in F(\varepsilon)$ for any $\sigma\in\mathrm{Gal(L/F)}$, and hence $\mathrm{Tr}_{L/F(\alpha)}(\varepsilon)\in F(\varepsilon)$.
This impies that $\alpha\in F(\varepsilon)$ by (\ref{trace}).
Therefore, we achieve
\begin{eqnarray*}
F(\varepsilon)=F(\alpha, \varepsilon)=F(\alpha, \varepsilon-a\alpha+b\mathrm{Tr}_{L/F(\alpha)}(\beta))
=F(\alpha,b\nu\beta)=L.
\end{eqnarray*}
\end{proof}

\begin{theorem}\label{generator}
Let $\ell$ and $p$ be odd primes, $n=({\ell-1})/{2}$ and $\mu\in\mathbb{Z}_{>0}$. 
And, let $\mathbf{r}$, $\mathbf{s}$, $z_\ell$ and $A_\ell(\mathbf{r},\mathbf{s})$ be as above. 
Assume that $z_\ell$ is neither a zero nor a pole of $~\displaystyle\prod_{i=1}^{n+1} \Phi_{[\mathbf{r},\mathbf{s};\mu,i]}(z)_p$. 
If $p\nmid \ell h_\ell^+ n\cdot \det\big(A_\ell(\mathbf{r},\mathbf{s})\big)$, then we have
\begin{equation}\label{odd}
K_{2\mu-1-\alpha}=k_{2\mu-1-\alpha}\left(\sum_{i=1}^{n+1}\Phi_{[\mathbf{r},\mathbf{s};\mu,i]}(z_\ell)_p^{p^{\alpha}}\right)
\end{equation}
for $\alpha=0,1,\ldots, \mu-1$.
Moreover, if $\dim_{\mathbb{Z}/p\mathbb{Z}}(H_1/S_2)=n-1$ then we have
\begin{equation}\label{ray class}
k_{2\mu-\alpha}=k_{\mu}\left(\sum_{i=1}^{n+1}\Phi_{[\mathbf{r},\mathbf{s};\mu,i]}(z_\ell)_p^{p^{\alpha}}\right)
\end{equation}
for $\alpha=0,1,\ldots, \mu-1$.
\end{theorem}
\begin{proof}
Let $\mathbf{x}_{\mu,i}=r_i/p^\mu$ and $\mathbf{y}_{\mu,i}=s_i/p^\mu$ for each $\mu$ and $i$.
Then $\mathbf{x}_{\mu,i}=\mathbf{x}_{1,i}/p^{\mu-1}$ and $\mathbf{y}_{\mu,i}=\mathbf{y}_{1,i}/p^{\mu-1}$.
By Lemma \ref{phi-class}, $\Phi_{[\mathbf{r},\mathbf{s};\mu,i]}(z_\ell)_p^{p^{\alpha}}\in K_{2\mu-1-\alpha}$.
It then follows from the formula (\ref{formula}) that for $1\leq i,j\leq n+1$
\begin{equation*}
\big(\Phi_{[\mathbf{r},\mathbf{s};\mu,i]}(z_\ell)_p^{p^\alpha}\big)^{\big(\frac{K_{2\mu-1-\alpha}/k}{(\O_{2\mu-1-\alpha,j})}\big)}
=e(-p^{2\mu-1}\cdot{^t}\mathbf{x}_{\mu,i}\mathbf{y}_{\mu,i})e\Big(\frac{{^t}\mathbf{x}_{\mu,i} b_{2\mu-1,j}-{^t}a_{2\mu-1,j}\mathbf{y}_{\mu,i}}{2}\Big)\Phi_{[\mathbf{r},\mathbf{s};\mu,i]}(z_\ell)_p^{p^\alpha},
\end{equation*}
where $a_{2\mu-1,j},b_{2\mu-1,j}\in 2p^{\mu-1}\mathbb{Z}^n$ such that
\begin{eqnarray*}
\left[\begin{matrix}
a_{2\mu-1,j}\\
b_{2\mu-1,j}
\end{matrix}\right]
&=&2p^{2\mu-1}\cdot{^t}h\big(\V^+(\Z^j)\big)\left[\begin{matrix}\mathbf{x}_{\mu,i}\\ \mathbf{y}_{\mu,i}\end{matrix}\right]\\
&=&2p^{\mu}\cdot{^t}h\big(\V^+(\Z^j)\big)\left[\begin{matrix}\mathbf{x}_{1,i}\\ \mathbf{y}_{1,i}\end{matrix}\right]\\
&=&p^{\mu-1}\left[\begin{matrix}a_{1,j}\\b_{1,j}\end{matrix}\right].
\end{eqnarray*}
And, we ensure
\begin{equation}\label{formula2}
\begin{array}{lll}
\big(\Phi_{[\mathbf{r},\mathbf{s};\mu,i]}(z_\ell)_p^{p^\alpha}\big)^{\big(\frac{K_{2\mu-1-\alpha}/k}{(\O_{2\mu-1-\alpha,j})}\big)}
&=&\displaystyle e(-p\cdot{^t}\mathbf{x}_{1,i}\mathbf{y}_{1,i})e\Big(\frac{{^t}\mathbf{x}_{1,i} b_{1,j}-{^t}a_{1,j}\mathbf{y}_{1,i}}{2}\Big)\Phi_{[\mathbf{r},\mathbf{s};\mu,i]}(z_\ell)_p^{p^\alpha}\vspace{0.1cm}\\
&=&\displaystyle\Z_p^{c_{ij}}\Phi_{[\mathbf{r},\mathbf{s};\mu,i]}(z_\ell)_p^{p^\alpha},
\end{array}
\end{equation}
where $c_{ij}$ is the $(i,j)^\mathrm{th}$ entry of $A_\ell(\mathbf{r},\mathbf{s})$.
Here, we observe that for $1\leq i \leq n+1$ there exists $\gamma_i\in k_{2\mu-1-\alpha}\big(\Phi_{[\mathbf{r},\mathbf{s};\mu,1]}(z_\ell)_p^{p^\alpha}, \Phi_{[\mathbf{r},\mathbf{s};\mu,2]}(z_\ell)_p^{p^\alpha}, \ldots, \Phi_{[\mathbf{r},\mathbf{s};\mu,n+1]}(z_\ell)_p^{p^\alpha}\big)$ with
\begin{equation*}
\gamma_i^{\big(\frac{K_{2\mu-1-\alpha}/k}{(\O_{2\mu-1-\alpha,j})}\big)}=
\left\{ \begin{array}{ll}
\Z_p\gamma_i & \textrm{if~ $j = i$}\\
\gamma_i & \textrm{if $j\neq i$}\\
\end{array} \right. 
\end{equation*}
because $p\nmid \det(A_\ell(\mathbf{r},\mathbf{s}))$.
Since $\big|\mathrm{Gal}(K_{2\mu-1-\alpha}/k_{2\mu-1-\alpha})\big|\leq p^{n+1}$ by (\ref{Galois}), 
we deduce 
\begin{equation}\label{generators}
\begin{array}{lll}
K_{2\mu-1-\alpha}&=&k_{2\mu-1-\alpha}\big(\gamma_1,\gamma_2,\ldots,\gamma_{n+1}\big)\\
&=&k_{2\mu-1-\alpha}\left(\Phi_{[\mathbf{r},\mathbf{s};\mu,1]}(z_\ell)_p^{p^{\alpha}}, \Phi_{[\mathbf{r},\mathbf{s};\mu,2]}(z_\ell)_p^{p^{\alpha}}, \ldots, \Phi_{[\mathbf{r},\mathbf{s};\mu,n+1]}(z_\ell)_p^{p^{\alpha}}\right).
\end{array}
\end{equation}
Now, let $L_i=k_{2\mu-1-\alpha}\left(\Phi_{[\mathbf{r},\mathbf{s};\mu,1]}(z_\ell)_p^{p^{\alpha}}, \Phi_{[\mathbf{r},\mathbf{s};\mu,2]}(z_\ell)_p^{p^{\alpha}}, \ldots, \Phi_{[\mathbf{r},\mathbf{s};\mu,i]}(z_\ell)_p^{p^{\alpha}}\right)$ for each $1\leq i\leq n+1$.
Suppose that $L_m=k_{2\mu-1-\alpha}\left(\sum_{i=1}^{m}\Phi_{[\mathbf{r},\mathbf{s};\mu,i]}(z_\ell)_p^{p^{\alpha}}\right)$ for some integer $1\leq m\leq n$.
Note that $[L_{m+1}:L_m]=p$ and
\begin{equation*}
\mathrm{Tr}_{L_{m+1}/L_m}(\Phi_{[\mathbf{r},\mathbf{s};\mu,m+1]}(z_\ell)_p^{p^{\alpha}})=\sum_{j=0}^{p-1}\Z_p^j \Phi_{[\mathbf{r},\mathbf{s};\mu,m+1]}(z_\ell)_p^{p^{\alpha}}=0,
\end{equation*}
due to the fact $\sum_{j=0}^{p-1}\Z_p^j=0$.
Hence if we take $a=1$ and $b={1}/{p}$, we achieve by Lemma \ref{primitive}
\begin{equation*}
L_{m+1}=k_{2\mu-1-\alpha}\left(\sum_{i=1}^{m}\Phi_{[\mathbf{r},\mathbf{s};\mu,i]}(z_\ell)_p^{p^{\alpha}}, \Phi_{[\mathbf{r},\mathbf{s};\mu,m+1]}(z_\ell)_p^{p^{\alpha}}\right)=k_{2\mu-1-\alpha}\left(\sum_{i=1}^{m+1}\Phi_{[\mathbf{r},\mathbf{s};\mu,i]}(z_\ell)_p^{p^{\alpha}}\right).
\end{equation*}
Therefore (\ref{odd}) is proved by induction and (\ref{generators}).
\par
If $\dim_{\mathbb{Z}/p\mathbb{Z}}(H_1/S_2)=n-1$, then by the proof of Corollary \ref{unit2} we get 
\begin{equation*}
|\mathrm{Gal}(k_{2\mu-\alpha}/k_{2\mu-1-\alpha})|=p^{n+1}.
\end{equation*}
Since $\big|\mathrm{Gal}(K_{2\mu-1-\alpha}/k_{2\mu-1-\alpha})\big|= p^{n+1}$ and $K_{2\mu-1-\alpha}\subset k_{2\mu-\alpha}$, we conclude $K_{2\mu-1-\alpha}=k_{2\mu-\alpha}$.
Hence, by (\ref{generators})
\begin{eqnarray*}
k_{2\mu-\alpha}&=&k_{2\mu-1-\alpha}\big(\Phi_{[\mathbf{r},\mathbf{s};\mu,1]}(z_\ell)_p^{p^{\alpha}}, \Phi_{[\mathbf{r},\mathbf{s};\mu,2]}(z_\ell)_p^{p^{\alpha}}, \ldots, \Phi_{[\mathbf{r},\mathbf{s};\mu,n+1]}(z_\ell)_p^{p^{\alpha}}\big)\\
&=&k_{\mu}\big(\Phi_{[\mathbf{r},\mathbf{s};\mu,1]}(z_\ell)_p^{p^{\alpha}}, \Phi_{[\mathbf{r},\mathbf{s};\mu,2]}(z_\ell)_p^{p^{\alpha}}, \ldots, \Phi_{[\mathbf{r},\mathbf{s};\mu,n+1]}(z_\ell)_p^{p^{\alpha}}\big)
\end{eqnarray*}
for  $\alpha=0, 1, \ldots, \mu-1$.
Let $L_{i}'=k_{\mu}\left(\Phi_{[\mathbf{r},\mathbf{s};\mu,1]}(z_\ell)_p^{p^{\alpha}}, \Phi_{[\mathbf{r},\mathbf{s};\mu,2]}(z_\ell)_p^{p^{\alpha}}, \ldots, \Phi_{[\mathbf{r},\mathbf{s};\mu,i]}(z_\ell)_p^{p^{\alpha}}\right)$ for each $1\leq i\leq n+1$.
Suppose that $L_{m}'=k_{\mu}\left(\sum_{i=1}^{m}\Phi_{[\mathbf{r},\mathbf{s};\mu,i]}(z_\ell)_p^{p^{\alpha}}\right)$ for some integer $1\leq m\leq n$.
Then we have
\begin{eqnarray*}
\mathrm{Tr}_{L_{m+1}/L_{m}'}(\Phi_{[\mathbf{r},\mathbf{s};\mu,m+1]}(z_\ell)_p^{p^{\alpha}})&=&
\mathrm{Tr}_{L_{m}/L_{m}'}\big(\mathrm{Tr}_{L_{m+1}/L_{m}}(\Phi_{[\mathbf{r},\mathbf{s};\mu,m+1]}(z_\ell)_p^{p^{\alpha}})\big)~(=0)\\
&=&\mathrm{Tr}_{L_{m+1}'/L_{m}'}\big(\mathrm{Tr}_{L_{m+1}/L_{m+1}'}(\Phi_{[\mathbf{r},\mathbf{s};\mu,m+1]}(z_\ell)_p^{p^{\alpha}})\big)\\
&=&[L_{m+1} : L_{m+1}']\cdot \mathrm{Tr}_{L_{m+1}'/L_{m}'}(\Phi_{[\mathbf{r},\mathbf{s};\mu,m+1]}(z_\ell)_p^{p^{\alpha}}),
\end{eqnarray*}
and so $\mathrm{Tr}_{L_{m+1}'/L_{m}'}(\Phi_{[\mathbf{r},\mathbf{s};\mu,m+1]}(z_\ell)_p^{p^{\alpha}})=0$.
Therefore, by Lemma \ref{primitive}
\begin{equation*}
L_{m+1}'=k_{\mu}\left(\sum_{i=1}^{m}\Phi_{[\mathbf{r},\mathbf{s};\mu,i]}(z_\ell)_p^{p^{\alpha}}, \Phi_{[\mathbf{r},\mathbf{s};\mu,m+1]}(z_\ell)_p^{p^{\alpha}}\right)=k_{\mu}\left(\sum_{i=1}^{m+1}\Phi_{[\mathbf{r},\mathbf{s};\mu,i]}(z_\ell)_p^{p^{\alpha}}\right)
\end{equation*}
and (\ref{ray class}) is proved again by induction.
\end{proof}

Although we omit in the above theorem the case where $p$ divides $\det(A_\ell(\mathbf{r},\mathbf{s}))$, by utilizing Theorem \ref{dimension} we might find suitable generators of $K_\mu$ over $k_\mu$ for each $\mu\in\mathbb{Z}_{>0}$.
\par
Let $\mathbf{r_0}=[r_i]_{1\leq i \leq n+1}$ and $\mathbf{s_0}=[s_i]_{1\leq i \leq n+1}$ where 
$r_i={^t}[1~~0~~\cdots~~0]\in\mathbb{Z}^n$ and $s_i={^t}\big[(s_i)_j\big]_{1\leq j\leq n}\in\mathbb{Z}^n$ for $1\leq i\leq n+1$ with
\begin{displaymath}
(s_i)_j = \left\{ \begin{array}{ll}
1~ & \textrm{if~ $j < i$}\\
0~ & \textrm{otherwise}.\\
\end{array} \right.
\end{displaymath}
Here we observe that $\Phi_{[\mathbf{r_0},\mathbf{s_0};\mu,i]}(z)_p$ is not identically zero for all $\mu$ and $i$ by Proposition \ref{theta-constant}.

\begin{corollary}\label{generator2}
Let $\ell$ and $p$ be odd primes and $z_\ell$ be as above.
Put $n=({\ell-1})/{2}$ and $\mu\in\mathbb{Z}_{>0}$.
Further, we assume that $z_\ell$ is not a zero of $~\displaystyle\prod_{i=1}^{n+1} \Phi_{[\mathbf{r_0},\mathbf{s_0};\mu,i]}(z)_p$.
\begin{itemize}
\item[\textup{(i)}] Let $\ell=7$. 
If $p\neq 3,7$, then we have
\begin{equation}\label{case1}
K_{2\mu-1-\alpha}=k_{2\mu-1-\alpha}\left(\sum_{i=1}^4\Phi_{[\mathbf{r_0},\mathbf{s_0};\mu,i]}(z_7)_p^{p^\alpha}\right)
\end{equation}
for $0\leq\alpha\leq \mu-1$.
\item[\textup{(ii)}] Let $\ell=11$.
If $p\neq 3,5,11$, then we have
\begin{equation}\label{case2-1}
K_{2\mu-1-\alpha}=k_{2\mu-1-\alpha}\left(\sum_{i=1}^6\Phi_{[\mathbf{r_0},\mathbf{s_0};\mu,i]}(z_{11})_p^{p^\alpha}\right)
\end{equation}
for $0\leq\alpha\leq \mu-1$.
And, if $p=3$ then we get
\begin{equation}\label{case2-2}
K_{2\mu-1-\alpha}=k_{2\mu-1-\alpha}\left(\sum_{i=1}^5\Phi_{[\mathbf{r_0},\mathbf{s_0};\mu,i]}(z_{11})_3^{3^\alpha}\right).
\end{equation}
\item[\textup{(iii)}] Let $\ell=13$.
If $p\neq 3,5,13$, then we have
\begin{equation}\label{case3-1}
K_{2\mu-1-\alpha}=k_{2\mu-1-\alpha}\left(\sum_{i=1}^7\Phi_{[\mathbf{r_0},\mathbf{s_0};\mu,i]}(z_{13})_p^{p^\alpha}\right)
\end{equation}
for $0\leq\alpha\leq \mu-1$.
And, if $p=5$ then we obtain
\begin{equation}\label{case3-2}
K_{2\mu-1-\alpha}=k_{2\mu-1-\alpha}\left(\sum_{i=1}^6\Phi_{[\mathbf{r_0},\mathbf{s_0};\mu,i]}(z_{13})_5^{5^\alpha}\right).
\end{equation}
\item[\textup{(iv)}] Let $\ell=5$. 
Then we have
\begin{equation}\label{case4}
K_{2\mu-1-\alpha}=k_{2\mu-1-\alpha}\left(\Z_{p^{2\mu-\alpha}}+ \Phi_{[\mathbf{r_0},\mathbf{s_0};\mu,1]}(z_{5})_p^{p^\alpha}+\Phi_{[\mathbf{r_0},\mathbf{s_0};\mu,3]}(z_{5})_p^{p^\alpha}\right)
\end{equation}
for $0\leq\alpha\leq \mu-1$.
\end{itemize}
\end{corollary}

\begin{proof}
Let the matrix $M_\ell(p)$ be as in Lemma \ref{independence}.
Note that $h_\ell^+=1$ for $\ell\leq 67$ (\cite[p.352]{Washington}).
We can show that $z_\ell$ is not a pole of $~\displaystyle\prod_{i=1}^{n+1} \Phi_{[\mathbf{r_0},\mathbf{s_0};\mu,i]}(z)_p$ for $\ell\leq 17$ by utilizing the \texttt{RiemannTheta} command in Maple.
\begin{itemize}
\item[(i)] Using the formula (\ref{formula}) we can find $\Phi_{[\mathbf{r_0},\mathbf{s_0};1,i]}(z_7)_p^{\big(\frac{K_1/k}{(\O_{1,j})}\big)}$ for each $1\leq i,j\leq 4$ as follows:
\begin{equation*}
\bbordermatrix{%
 & \big(\frac{K_1/k}{(\O_{1,1})}\big) & \big(\frac{K_1/k}{(\O_{1,2})}\big)&\big(\frac{K_1/k}{(\O_{1,3})}\big)&\big(\frac{K_1/k}{(\O_{1,4})}\big) \cr
\Phi_{[\mathbf{r_0},\mathbf{s_0};1,1]}(z_7)_p & -1 & -1 & -1 & 1\cr
\Phi_{[\mathbf{r_0},\mathbf{s_0};1,2]}(z_7)_p & -2 & -4 & -2 & 0\cr
\Phi_{[\mathbf{r_0},\mathbf{s_0};1,3]}(z_7)_p & 0  & -10& -4 & 2 \cr
\Phi_{[\mathbf{r_0},\mathbf{s_0};1,4]}(z_7)_p & -3 & -13&-11 & 9 }
=A_7(\mathbf{r_0},\mathbf{s_0}).
\end{equation*}
Since $\det(A_7(\mathbf{r_0},\mathbf{s_0}))=2^6$ is prime to $p$, (\ref{case1}) is true by Theorem \ref{generator}.
\item[(ii)] First, suppose that $p\neq 3,5,11$. 
In a similar way as in (i) we obtain
\begin{equation*}
A_{11}(\mathbf{r_0},\mathbf{s_0})=
\left[\begin{matrix}
 -1 & -1 & -1 & -1 & -1 & 1\cr
 -2 & -4 & -2 & -4 & -2 & 0\cr
 -4 & -6 & -4 & -10 & 0 & -2 \cr
 -7 & -3 &-11 & -21 & -3& 1 \cr
 -7 & -5 &-25 & -29 & -1& -1 \cr
 -10& -2 &-48 & -34 & -6& 4 \cr
\end{matrix}\right].
\end{equation*}
Since $\det(A_{11}(\mathbf{r_0},\mathbf{s_0}))=2^7\cdot 3\cdot 5^2$ is prime to $p$, we get (\ref{case2-1}) by Theorem \ref{generator}.
If $p=3$, then the rank of $M_{11}(3)$ is equal to 5. 
Since $p\nmid 11\cdot 5$, by Theorem \ref{dimension} we deduce Gal$(K_\mu/k_\mu)\cong (\mathbb{Z}/3\mathbb{Z})^5$ for all $\mu\in\mathbb{Z}_{>0}$. 
And we observe that the determinant of the matrix
\begin{equation*}
\bbordermatrix{%
 & \big(\frac{K_1/k}{(\O_{1,2})}\big)&\big(\frac{K_1/k}{(\O_{1,3})}\big)&\big(\frac{K_1/k}{(\O_{1,4})}\big)&\big(\frac{K_1/k}{(\O_{1,5})}\big)&\big(\frac{K_1/k}
{(\O_{1,6})}\big) \cr
\Phi_{[\mathbf{r_0},\mathbf{s_0};1,1]}(z_{11})_3 & -1 & -1 & -1 & -1 & 1\cr
\Phi_{[\mathbf{r_0},\mathbf{s_0};1,2]}(z_{11})_3 & -4 & -2 & -4 & -2 & 0\cr
\Phi_{[\mathbf{r_0},\mathbf{s_0};1,3]}(z_{11})_3 & -6 & -4 & -10 & 0 & -2 \cr
\Phi_{[\mathbf{r_0},\mathbf{s_0};1,4]}(z_{11})_3 & -3 &-11 & -21 & -3& 1 \cr
\Phi_{[\mathbf{r_0},\mathbf{s_0};1,5]}(z_{11})_3 & -5 &-25 & -29 & -1& -1 \cr
}
\end{equation*}
is equal to $2^5\cdot5\cdot11$ which is prime to 3. 
Using Lemma \ref{primitive} and (\ref{formula2}) we can conclude (\ref{case2-2}).
\item[(iii)] First, suppose that $p\neq 3,5,13$. 
Then we derive
\begin{equation*}
A_{13}(\mathbf{r_0},\mathbf{s_0})=
\left[\begin{matrix}
 -1 & -1 & -1 & -1 & -1 & -1 & 1\cr
 -2 & -4 & -2 & -4 & -2 & -4 & 2\cr
  0 & -10& -4 & -6 & -4 & -10& 8 \cr
  5 & -19& -7 & -11& -11& -13& 11 \cr
  7 & -35&-13 & -13& -19& -15& 13 \cr
  2 & -60& -18& -8 & -32& -22& 20 \cr
 -10& -84&-28 &  2 & -54& -22& 20 \cr
\end{matrix}\right].
\end{equation*}
Since $\det(A_{13}(\mathbf{r_0},\mathbf{s_0}))=-2^{12}\cdot 5^2$ is prime to $p$, we have (\ref{case3-1}) again by Theorem \ref{generator}.
If $p=5$, then the rank of $M_{13}(5)$ is equal to 6. 
Since $p\nmid 13\cdot 6$, it follow from Theorem \ref{dimension} that
Gal$(K_\mu/k_\mu)\cong (\mathbb{Z}/5\mathbb{Z})^6$ for all $\mu\in\mathbb{Z}_{>0}$.
Observe that the determinant of the matrix
\begin{equation*}
\bbordermatrix{%
 & \big(\frac{K_1/k}{(\O_{1,2})}\big)&\big(\frac{K_1/k}{(\O_{1,3})}\big)&\big(\frac{K_1/k}{(\O_{1,4})}\big)&\big(\frac{K_1/k}{(\O_{1,5})}\big)&\big(\frac{K_1/k}{(\O_{1,6})}\big)&\big(\frac{K_1/k}
{(\O_{1,7})}\big) \cr
\Phi_{[\mathbf{r_0},\mathbf{s_0};1,1]}(z_{13})_5 & -1 & -1 & -1 & -1 & -1 & 1\cr
\Phi_{[\mathbf{r_0},\mathbf{s_0};1,2]}(z_{13})_5 & -4 & -2 & -4 & -2 & -4 & 2\cr
\Phi_{[\mathbf{r_0},\mathbf{s_0};1,3]}(z_{13})_5 & -10& -4 & -6 & -4 & -10& 8\cr
\Phi_{[\mathbf{r_0},\mathbf{s_0};1,4]}(z_{13})_5 & -19& -7 & -11& -11& -13& 11 \cr
\Phi_{[\mathbf{r_0},\mathbf{s_0};1,5]}(z_{13})_5 & -35&-13 & -13& -19& -15& 13 \cr
\Phi_{[\mathbf{r_0},\mathbf{s_0};1,6]}(z_{13})_5 & -60& -18& -8 & -32& -22& 20 \cr
}
\end{equation*}
is equal to $-2^7\cdot31$ which is prime to 5. 
Using Lemma \ref{primitive} and (\ref{formula2}) we can deduce (\ref{case3-2}).

\item[(iv)] In this case $\det(A_5(\mathbf{r_0},\mathbf{s_0}))=0$ so that we should find another generators of $K_{2\mu-1-\alpha}$ over $k_{2\mu-1-\alpha}$.
By \cite[p.316]{Komatsu}, $H_{2\mu-1-\alpha}/S_{2\mu-\alpha}$ is generated by real units of $k$ for any odd prime $p$.
Using the idea in the proof of Theorem \ref{dimension}, one can show that $(\V^*(S_{2\mu-1-\alpha})\cap H_{2\mu-1-\alpha})/S_{2\mu-\alpha} =\{0\}$, and so $\Phi_{[\mathbf{r_0},\mathbf{s_0};\mu,i]}(z_{5})_p^{p^\alpha}\in K_{2\mu-1-\alpha}$ for $1\leq i\leq 3$.
Note that $\Z_{p^{2\mu-\alpha}}\in\mathcal{F}_{2p^{2\mu-\alpha}}$ is $R_{2p^{2\mu-\alpha}}$-invariant, hence $\Z_{p^{2\mu-\alpha}}\in K_{2\mu-1-\alpha}$ by Proposition \ref{reciprocity}. 
Since $N_{k/\mathbb{Q}}(\O_{2\mu-1-\alpha,j})\equiv 1-2p^{2\mu-1-\alpha}\pmod{2p^{2\mu-\alpha}}$ for $1\leq j \leq 3$, we get
\begin{equation}\label{zeta}
\Z_{p^{2\mu-\alpha}}^{\big(\frac{K_{2\mu-1-\alpha}/k}{(\O_{2\mu-1-\alpha,j})}\big)}=\Z_p^{-2}\Z_{p^{2\mu-\alpha}}\quad\textrm{ for $1\leq j\leq 3$}.
\end{equation}
Now, observe that the determinant of the matrix
\begin{equation*}
\bbordermatrix{%
 & \big(\frac{K_1/k}{(\O_{1,1})}\big) & \big(\frac{K_1/k}{(\O_{1,2})}\big)&\big(\frac{K_1/k}{(\O_{1,3})}\big) \cr
\Z_{p^2} & -2 & -2 & -2 \cr
\Phi_{[\mathbf{r_0},\mathbf{s_0};1,1]}(z_5)_p & -1 & -1 &  1\cr
\Phi_{[\mathbf{r_0},\mathbf{s_0};1,3]}(z_5)_p & -4 & -6& 4  \cr
}
\end{equation*}
is equal to $-2^3$ which is prime to $p$.
Therefore, we obtain (\ref{case4}) by Lemma \ref{primitive}, (\ref{formula2}) and (\ref{zeta}).
\end{itemize}
\end{proof}
\begin{remark}
\begin{itemize}
\item[(i)] Especially when $\mu=1$, Corollary \ref{generator2} (iv) is reduced to Komatsu's work (\cite[Proposition 1]{Komatsu}) with a little different ingredients.
\item[(ii)]
The following table gives the prime factors of $\det(A_\ell(\mathbf{r_0},\mathbf{s_0}))$ for $\ell\leq 89 $.

\begin{tabular}{|c|p{13.0cm}|}
\hline
\multicolumn{1}{|c|}{$\ell$} & \multicolumn{1}{c|}{prime factors of $\det(A_\ell(\mathbf{r_0},\mathbf{s_0}))$} \\
\hline
3 & 2\\
5 & 0\\
7 & 2\\
11& 2, 3, 5\\
13& 2, 5\\
17& 2, 7, 17, 43\\
19& 2, 3, 36137\\
23& 2, 3, 11, 13, 29, 89, 241\\ 
29& 2, 3, 5, 13, 113, 58057291\\
31& 2, 3, 31, 109621, 1216387\\
37& 2, 5, 13, 37, 53, 109, 10138325056259\\
41& 2, 5, 11, 17, 41, 439, 1667, 166013, 203381\\
43& 2, 3, 19, 43, 211, 281345721890371109\\
47& 2, 5, 83, 139, 5323, 178481, 6167669171116393\\
53& 2, 3, 5, 139, 157, 1613, 4889, 1579367, 28153859844430949\\
59& 2, 3, 59, 233, 3033169, 1899468180409634452730252070517\\
61& 2, 5, 11, 13, 41, 1321, 1861, 1142941857599125232990619467569\\
67& 2, 3, 67, 683, 12739, 20857, 513881, 1858283767, 986862333655510350967\\
71& 2, 5, 7, 31, 79, 127, 1129, 79241, 122921, 68755411, 1190061671, 3087543529906501\\
73& 2, 7, 73, 79, 89, 16747, 134353, 5754557119657, 1150806776867233, 1190899  \\
79& 2, 5, 7, 13, 29, 53, 1427, 3847, 8191, 121369, 377911, 1842497, 51176893, 357204083, 32170088152177\\
83& 2, 3, 13, 17387, 279405653, 43059261982072584626787705301351, 8831418697, 758583423553 \\
89& 2, 17, 23, 89, 113, 313629821584641896139082338756559409, 4504769, 118401449, 22482210593 \\
\hline
\end{tabular}

\end{itemize}
\end{remark}

\bibliographystyle{amsplain}

\address{% First Author
Ja Kyung Koo\\
Department of Mathematical Sciences \\
KAIST \\
Daejeon 305-701 \\
Republic of Korea} {jkkoo@math.kaist.ac.kr}
%%%%%%%%%
\address{% Corresponding Author
Dong Sung Yoon\\
National Institute for Mathematical Sciences \\
Daejeon 305-811 \\
Republic of Korea} {dsyoon@nims.re.kr}

\end{document}